\documentclass{siamart1116}
\numberwithin{theorem}{section}

\usepackage{lipsum}
\usepackage{amsfonts}
\usepackage{graphicx}
\usepackage{epstopdf}
\usepackage{algorithmic}

\usepackage[T1]{fontenc}
\usepackage[UKenglish]{babel}%
\usepackage{amssymb, amsmath}
\usepackage{lmodern}
\usepackage{dsfont}
\usepackage{enumerate}
\usepackage[all]{xy}
\usepackage{mathtools}
\usepackage{appendix}
\usepackage{pgf,tikz}
\usepackage{tikz-cd}
\usetikzlibrary{arrows}
\usepackage{centernot}
\usepackage{verbatim}
\usepackage[normalem]{ulem}

\ifpdf
  \DeclareGraphicsExtensions{.eps,.pdf,.png,.jpg}
\else
  \DeclareGraphicsExtensions{.eps}
\fi

\usepackage{amsopn}

\newcommand{\sep}{,\,}

\DeclareFontFamily{U}{mathx}{\hyphenchar\font45}
\DeclareFontShape{U}{mathx}{m}{n}{
      <5> <6> <7> <8> <9> <10>
      <10.95> <12> <14.4> <17.28> <20.74> <24.88>
      mathx10
      }{}
\DeclareSymbolFont{mathx}{U}{mathx}{m}{n}
\DeclareFontSubstitution{U}{mathx}{m}{n}
\DeclareMathAccent{\widecheck}{0}{mathx}{"71}
\DeclareMathAccent{\wideparen}{0}{mathx}{"75}

\usepackage{mathrsfs}  

\newcommand{\grad}{\nabla}
\newcommand{\vc}[1]{\mathbf{#1}}

\newcommand{\N}{\mathbb{N}}
\newcommand{\R}{\mathbb{R}}

\newcommand{\NB}{\mathcal{H}}

\renewcommand{\S}{\mathbb{S}}

\newcommand{\ps}[2]{\left\langle #1 , #2 \right\rangle}
\newcommand{\norm}[1]{\| #1 \|}
\newcommand{\nnorm}[1]{{\vert\kern-0.25ex\vert\kern-0.25ex\vert #1 
\vert\kern-0.25ex\vert\kern-0.25ex\vert}}
\renewcommand{\bar}[1]{\overline{#1}}

\renewcommand{\epsilon}{\varepsilon}

\newcommand{\bal}{\boldsymbol{\alpha}}

\newcommand{\bbe}{\boldsymbol{\beta}}

\newcommand{\bt}{\boldsymbol{\theta}}

\newcommand{\cwl}{\widehat{\operatorname{cw}}_{\bb}}
\newcommand{\cwu}{\widecheck{\operatorname{cw}}_{\bb}}

\renewcommand{\O}{\mathcal O}
\newcommand{\Dn}{\Delta_{++}^d}

\newcommand{\bxi}{\boldsymbol{\xi}}
\newcommand{\blam}{\boldsymbol{\lambda}}

\newcommand{\ba}{\vc a}
\newcommand{\bb}{\vc b}

\newcommand{\be}{\vc e}

\newcommand{\bx}{\vc x}
\newcommand{\bxb}{\vc x}

\newcommand{\by}{\vc y}
\newcommand{\byb}{\vc y}

\newcommand{\bz}{\vc z}
\newcommand{\bzb}{\vc z}

\newcommand{\bv}{\vc v}
\newcommand{\bvb}{\vc v}

\newcommand{\bu}{\vc u}
\newcommand{\bub}{\vc u}

\newcommand{\bw}{\vc w}
\newcommand{\bwb}{\vc w}

\newcommand{\bs}{\vc s}

\newcommand{\s}{s}

\newcommand{\A}{\mathcal{A}}

\newcommand{\Bmin}[2]{\text{\delimitershortfall=0pt
\delimiterfactor=1001 $\mathbf{\mathfrak{m}}\!\left(#1\middle/#2\right)$}}
\newcommand{\Bmax}[2]{\text{\delimitershortfall=0pt
\delimiterfactor=1001 $\mathfrak{M}\!\left(#1\middle/#2\right)$}}

\newcommand{\bmini}[3]{\text{\delimitershortfall=0pt
\delimiterfactor=1001 $\mathfrak{m}_{#1}\!\left(#2\middle/#3\right)$}}
\newcommand{\bmaxi}[3]{\text{\delimitershortfall=0pt
\delimiterfactor=1001 $\mathfrak{M}_{#1}\!\left(#2\middle/#3\right)$}}

\newcommand{\and}{\quad\text{and}\quad}
\newcommand{\andd}{\qquad\text{and}\qquad}

\newcommand{\E}{\mathcal{E}}

\newcommand{\kone}{\mathcal{K}}

\newcommand{\lek}{\leq_{\kone}}
\newcommand{\lekk}{\lneq_{\kone}}
\newcommand{\lekkk}{<_{\kone}}

\newcommand{\G}{\mathcal{G}}

\newcommand{\krog}{\otimes}

\newcommand{\J}{\mathcal J}

\newcommand{\I}{\mathcal I}
\newcommand{\saufzero}{\setminus\{0\}}

\renewcommand{\t}{\tilde}

\newcommand{\multihomo}{multi-homogeneous}

\newcommand{\NF}{G}

\newcommand{\ones}{\mathbf{1}}

\newcommand{\sauf}{\setminus}

\newcommand{\bphi}{\boldsymbol{\phi}}
\newcommand{\bphib}{\boldsymbol{\phi}}
\newcommand{\Sn}{\S}

\newcommand{\pmi}[1]{\bx^{#1}}

\usepackage{soul} 

\def\red#1{\color{red}#1 \color{black}}

\newsiamthm{thm}{Theorem}
\newsiamthm{lem}{Lemma}
\newsiamthm{prop}{Proposition}
\newsiamthm{cor}{Corollary}
\newsiamremark{rmq}{Remark}
\newsiamremark{ex}{Example}
\newsiamremark{defi}{Definition}
\newsiamremark{ass}{Assumption}

\newcommand{\TheTitle}{The Perron-Frobenius Theorem for Multi-homogeneous Mappings} 
\newcommand{\TheAuthors}{A. Gautier, F. Tudisco, and M. Hein}

\headers{Multi-homogeneous Perron-Frobenius theorem}{\TheAuthors}

\title{{\TheTitle}\thanks{Author's accepted version: this is the peer-reviewed version of this manuscript, which is now   published on SIAM Journal on Matrix Analysis and Applications \url{https://doi.org/10.1137/18M1165037}.
\funding{This work has been funded by the ERC starting grant ``NOLEPRO'', no.\ 307793. The work of F.T. was funded by the European Union's Horizon 2020 research and innovation programme under the MarieSk\l odowska-Curie individual fellowship ``MAGNET'' grant agreement no.\ 744014.}}}

\author{
  Antoine Gautier\thanks{Department of Mathematics and Computer Science, Saarland University, 66041 Saarbr\"{u}cken, Germany
    (\email{ag@cs.uni-saarland.de},\email{hein@math.uni-sb.de}).}
  \and
 Francesco Tudisco\thanks{Department of Mathematics and Statistics, University of Strathclyde, G11XH Glasgow, UK (\email{f.tudisco@strath.ac.uk}).}
  \and
 Matthias Hein\footnotemark[2]
}

\newcommand{\Sphi}{\S^{\bphib}}
\newcommand{\bi}{\mathbf{i}}

\usepackage{mathtools}

{\color{blue} }%
{\color{black} }
\usepackage{soul}

\begin{document}

\maketitle

\begin{abstract}
  The Perron-Frobenius theory for nonnegative matrices has been generalized to order-preserving homogeneous mappings on a cone and more recently to nonnegative multilinear forms. 
   We unify both approaches by introducing the concept of order-preserving multi-homogeneous mappings, their associated nonlinear spectral problems and spectral radii. We show several Perron-Frobenius type results for these mappings addressing existence, uniqueness and maximality of nonnegative and positive eigenpairs. We prove a Collatz-Wielandt principle and other characterizations of the spectral radius and analyze the convergence of iterates of these mappings towards their unique positive eigenvectors. On top of providing a new extension of the nonlinear Perron-Frobenius theory to the multi-dimensional case, our contribution poses the basis for several improvements and a deeper understanding of the current spectral theory for nonnegative tensors. In fact, in recent years, important results have been obtained by recasting certain spectral equations for multilinear forms in terms of homogeneous maps, however as our approach is more adapted to such problems, these results can be further refined and improved by employing our new multi-homogeneous setting. 
\end{abstract}



\section{Introduction}
The classical Perron-Frobenius theory addresses properties such as existence, uniqueness and maximality of eigenvectors and eigenvalues of matrices with nonnegative entries. Two important generalizations of this theory arise in the study of eigenvectors of order-preserving homogeneous mappings defined on cones and in multilinear algebra where spectral problems involving nonnegative tensors are considered. In this work we consider a framework allowing the unified study of both directions by introducing the concept of order-preserving multi-homogeneous mappings. While some multi-homogeneous spectral problems can be reformulated in terms of standard homogeneous maps (see e.g. \cite{Fried}), the novel multi-homogeneous formulation allows us to go further and prove several results that either hold for a larger class of problems or that require weaker assumptions. In particular, we provide a notion of eigenvalue and spectral radius for multi-homogeneous 
mappings and prove several Perron-Frobenius type results. These results include the existence of a nonnegative eigenvector corresponding to the spectral radius, the existence and uniqueness of a positive maximal eigenvector, and a Collatz-Wielandt characterization of the spectral radius. Furthermore, we investigate the simplicity of the spectral radius and the convergence of the iterates of the mapping towards its unique positive eigenvector. The latter result  is particularly relevant from a computational viewpoint as it naturally gives rise to an efficient and general algorithm for the computation of the positive eigenvector, with a linear convergence rate. 


On the one side linear algebra can be seen as a special case of multilinear algebra, on the other side eigenvectors and eigenvalues of nonnegative matrices are a special case of those of order-preserving homogeneous mappings on $\R^n_+=\{\bu\in\R^n\colon u_i\geq 0, \ \forall i \}$. Following a similar analogy, the nonlinear Perron-Frobenius theory for homogeneous mappings is a special case of that for multi-homogeneous mappings and the study of spectral problems induced by nonnegative multi-linear forms is a special case of the study of eigenvectors of order-preserving multi-homogeneous mappings on the product $\R^{n_1}_+\times \ldots \times \R^{n_d}_+$. Therefore, when $d=1$, our results reduce to their counterparts in the existing linear and nonlinear finite dimensional Perron-Frobenius theories. However, when $d>1$, the use of the proposed multi-homogeneous setting allows us to improve and unify many results and definitions in the study of spectral problems induced by nonnegative multi-linear forms, as for instance the   $\ell^p$-eigenvector problem for (square) nonnegative tensors, the $\ell^{p,q}$-singular vector problem for nonnegative (rectangular) tensors and the $\ell^{p,q,r}$-singular vector problem for nonnegative tensors \cite{us,Yang2,Yang1}. In \cite{tensorpaper} we discuss several of these implications in detail.

In recent years, the nonlinear Perron-Frobenius theory and the spectral theory of nonnegative multi-linear forms have been successfully employed in a variety of applications ranging from signal processing \cite{Pillai2005ThePT} to low rank approximation \cite{Quynhn}, mathematical economics \cite{Gaubert} and dynamical systems \cite{7039628}. The use of the multi-homogeneous framework opens the avenue to several challenging applications too. For instance, the techniques proposed in this paper have recently inspired the first practicable algorithm for the training of a class of generalized polynomial neural networks to global optimality \cite{gautier2016globally}, and have been employed in network science in order to extend eigenvector-based centrality measures to multi-dimensional graphs \cite{frafra}.

The nonlinear Perron-Frobenius theory has been developed for order-preserving mappings on general cones. However, for the sake of simplicity and in order to make our ideas more transparent, we restrict ourself to cones of the form $\R^n_+$ and their Cartesian product. Nevertheless, we took special care to use as little as possible the particular structure of $\R^n_+$ in order to facilitate subsequent generalizations of our results to general cones. 

The paper is organized as follows: In Section \ref{mon}, we introduce and motivate the class of order-preserving multi-homogeneous mappings. We propose a way to define eigenvectors and eigenvalues for multi-homogeneous mappings.
Furthermore, we discuss characteristics of these mappings.
In Section \ref{contrasec}, we prove a contraction principle for our class of mappings in Theorem \ref{Banachcombi}. In particular, this theorem implies the existence and uniqueness of a positive eigenvector under very mild conditions. In Section \ref{existence_section}, we propose a generalized notion of spectral radius and prove, in Theorem \ref{weak_PF}, a weak form of the Perron-Frobenius theorem which implies the existence of a nonnegative eigenvector corresponding to the spectral radius. Then, we discuss a generalized notion of irreducibility allowing us to give, in Theorem \ref{existence_intro}, a sufficient condition for the existence of a positive eigenvector of nonexpansive mappings. In Section \ref{CW_M_U}, we prove a Collatz-Wielandt formula for the spectral radius (Theorem \ref{CW_intro}) and discuss the simplicity and the uniqueness of positive eigenvectors and their associated eigenvalues (Theorem \ref{simplicity}). Finally, in Section \ref{PM_section}, we discuss a method for computing the positive eigenvector of order-preserving multi-homogeneous mappings. The convergence of this method (with a linear rate) is discussed in Theorem \ref{PM}. For the sake of readability, in Sections  \ref{contrasec}\,--\,\ref{PM_section} we first state and discuss the main results and then we proceed with the proofs. For brevity, we shall  prove  only the results whose generalization from the homogeneous case is not straightforward.

\section{Motivation, overview and notation}\label{mon}
In this section we define and motivate the class of \textit{multi-homogeneous} mappings considered in this paper. We also introduce most  of the relevant notation that will be used throughout and we discuss a number of relatively simple but useful preliminary observations and results.

\subsection{Multi-homogeneous mappings}
First, let us introduce the class of mappings $\NB^d$. To this end, let  $n_1,\ldots,n_d$ be positive integers and consider the product space $V=\R^{n_1}\times\dots\times\R^{n_d}$. Within $V$, consider the product cone $\kone_{+}=\R^{n_1}_+\times \ldots \times \R^{n_d}_+$.
Let $F_i\colon \kone_+\to\R^{n_i}_+$, $i=1,\dots,d$, be continuous mappings and define $F\colon \kone_+ \to \kone_+$ as $F(\bx)=(F_1(\bx),\ldots,F_d(\bx))$. We say that $F$ is (positively) multi-homogeneous if there exists a $d\times d$ nonnegative matrix $A$ such that for every $\bx_j\in \R_+^{n_j}$ and every $\alpha_j\geq 0$, $j=1,\dots,d$, it holds
\begin{equation}\label{multihomodef}
F_i(\alpha_1\bx_1,\ldots,\alpha_d\bx_d)=\bigg(\prod_{j=1}^d \alpha_j^{A_{i,j}} \bigg)  \, F_i(\bx_1,\ldots,\bx_d)\qquad \forall i \in [d],
\end{equation}
where for an integer $n$ we let $[n]=\{1,\dots,n\}$. 

We refer to $A$ as the homogeneity matrix of $F$. 
When $d=1$, multi-homogeneous mappings boil down to standard homogeneous maps. To emphasize this property, it is useful to introduce the following matrix-power operation. 
For $\bal\in\R^n_{+}$ and a nonnegative matrix $A\in\R^{n\times n}$ define the vector  $\bal^A\in\R^n_+$ as
\begin{equation}\label{eq:a^A}
\bal^{A}= \Big(\prod_{k=1}^n\alpha_k^{A_{1,k}},\ldots,\prod_{k=1}^n\alpha_k^{A_{n,k}}\Big).
\end{equation}
With this notation,  \eqref{multihomodef} can be compactly written as
\begin{equation}\label{multihomodef_2}
F(\bal \otimes \bx) = \bal^A \otimes F(\bx)
\end{equation}
where, for $\bal\in \R^d$ and $\bx \in V$,  $\bal \krog \bx$ denotes the vector $\bal \krog \bx = (\alpha_1\bx_1,\ldots,\alpha_d\bx_d) \in V$. It is now immediate to note that, when $d=1$ and $A=p\in\R^{1\times 1}$, \eqref{multihomodef_2} boils down to $F(\alpha \bx) = \alpha^pF(\bx)$  which shows that every $p$-homogeneous mapping $F\colon \R^{n_1}\to \R^{n_1}$ is multi-homogeneous with homogeneity matrix $p\in \R$.

On $V$ we consider the partial ordering induced by $\kone_+$. If $\kone_{++}$ denotes the interior of $\kone_+$, we write $\bx\lek\bu$, $\bx\lekk \bu$, $\bx\lekkk\bu$ if $\bu-\bx\in\kone_+$, $\bu-\bx\in\kone_{+}\saufzero$ and $\bu-\bx\in\kone_{++}$, respectively. 
A multi-homogeneous mapping $F$ is said to be order-preserving if it preserves such ordering, that is for any $\bx,\by\in \kone_+$ we have 
\begin{equation}\label{orderpresdef}
 \bx \leq_\kone \by \quad \Longrightarrow \quad F(\bx)\leq_\kone F(\by)\, .
\end{equation}

Finally, we say that $F$ is nondegenerate if 
\begin{equation}\label{nondegdef}
F(\kone_{++})\subset \kone_{++} \qquad \text{and}\qquad A\R^d_{++}\subset \R^d_{++}
\end{equation}
where $A$ is the homogeneity matrix of $F$ and $\R^d_{++}$ is the set of positive vectors in $\R^d$ (the interior of $\R^d_+$). 
Overall, we define
\begin{defi}\label{defNBd}
Let $\NB^d$ denote the set of  multi-homogeneous order-preserving nondegenerate
mappings on $\kone_+$,  i.e. 
\begin{equation*}
\NB^d=\big\{ F\colon\kone_+\to \kone_+\ \big|\ F \text{ is continuous and satisfies } \eqref{multihomodef},\eqref{orderpresdef}, \eqref{nondegdef}\big\}.
\end{equation*}
For $F\in\NB^d$, we write $\A(F)$ to denote its homogeneity matrix $A$, as defined in \eqref{multihomodef_2}. 
\end{defi}

As discussed in the preface of \cite{NB}, the development of the nonlinear Perron-Frobenius theory strongly relies on the use of the Hilbert's projective metric combined with results from fixed point theory. In fact, for example, the following observation holds in the linear case: Let $M\in\R^{n\times n}$ be a matrix with $M\R^n_{++}\subset\R^n_{++}$, then
\begin{equation}\label{linear_birk} 
\mu(M\bx,M\by) \leq \mu(\bx,\by) \qquad \forall \bx,\by\in\R^n_{++},
\end{equation}
where $\mu\colon \R^n_{++}\times \R^n_{++}\to \R_+$ is the Hilbert metric defined as
\begin{equation}\label{eq:standard_H_metric}
\mu(\bx,\by) =\ln\Big(\max_{i,j\in[n]} \frac{x_i}{y_i}\frac{y_j}{x_j}\Big),
\end{equation}
In particular, it is known that for any norm $\norm{\cdot}$ on $\R^n$, the pair $\big(\{\bx\in\R^n_{++}\colon \norm{\bx}=1\},\mu\big)$ forms a complete metric space (see for instance Proposition 4.4 in \cite{Troyanov}) and so one can use results of fixed point theory to analyze the eigenvectors of $M$. This observation can be extended to a wider class of mappings, namely the class of mappings $F\colon\R^n_+\to\R^n_+$ which are positively $p$-homogeneous, order-preserving and leave $\R^n_{++}$ invariant. 
In fact, for this type of mappings it can then be shown that $p$ is a Lipschitz constant of $F$ with respect to $\mu$ (see for instance Theorem 3.1 in \cite{Bus73}). As we will show in Lemma \ref{contract}, a key property  of the Perron-Frobenius theory for order-preserving multi-homogeneous mappings is that this property can be generalized to mappings in $\NB^d$. More precisely, if $F\in\NB^d$, $A=\A(F)$ and there exists a positive eigenvector $\bb\in\R^d_{++}$ of $A^\top$, then 
\begin{equation}\label{Lipschitz}
\mu_{\bb}\big(F(\bx),F(\by)\big)\,\leq\, \rho(A)\,\mu_{\bb}(\bx,\by) \qquad \forall \bx,\by\in\kone_{++},
\end{equation}
where $\rho(A)$ is the spectral radius of $A$ and $\mu_{\bb}\colon\kone_{++}\times \kone_{++}\to \R_+$ is the weighted product metric defined as
\begin{equation*}
\mu_{\bb}\big((\bx_1,\ldots,\bx_d),(\by_1,\ldots,\by_d)\big)=\frac{1}{\|\bb\|_1}\sum_{i=1}^db_i\mu(\bx_i,\by_i).
\end{equation*}
Clearly, $\rho(A)=1$ if $F$ is linear and thus \eqref{linear_birk} is a special case of \eqref{Lipschitz}.

One may wonder why we do not identify $\kone_+$ with $\R^{n_1+\dots+n_d}_+$ and then consider the Hilbert metric on  $\R^{n_1+\dots+n_d}_{++}$ for the study of mappings in $\NB^d$. This is because,  as we will observe in Example \ref{tightmot}, there exist mappings $F\in \NB^d$ that are nonexpansive with respect to the weighted Hilbert and Thompson metrics, even though $F^k$ is expansive with respect to the Hilbert and Thompson metrics on  $\R^{n_1+\dots+n_d}_{++}$, for all $k\geq 1$.

Another example is given by the singular value problem for nonnegative tensors, considered for example in \cite{Fried} and \cite{us}. While the analysis carried out in those papers is based on a spectral problem for an order-preserving $1$-homogeneous mapping, as observed in \cite{tensorpaper}, spectral problems for tensors are naturally multi-homogeneous and the assumptions required by transforming them into a  homogeneous setting ($d=1$) are much more restrictive than the ones one gets by treating the problem in its original multi-homogeneous formulation. We refer to \cite{tensorpaper} for a detailed analysis of multi-homogeneous mappings associated with tensor spectral problems.

As, the Perron-Frobenius theorem is concerned with eigenvectors and eigenvalues, we propose a generalization of these objects in the context of multi-homogeneous mappings: 
\begin{defi}\label{eigenvectordef}
Let $F=(F_1,\ldots,F_d)\in\NB^d$. We say that $\bx=(\bx_1,\ldots,\bx_d)\in\kone_+$ is an eigenvector of $F$ if $\bx_i\neq 0$ for every $i=1,\ldots,d$ and there exists $\blam\in\R^d_+$ such that $F(\bx) = \blam \otimes \bx$, i.e. $F(\bx) = (\lambda_1 \bx_1,\ldots,\lambda_d\bx_d)$. The vector $\blam$ is an eigenvalue of $F$ corresponding to $\bx$. 
\end{defi}
We conclude the section with a few simple examples of multi-homogeneous maps. Let $M\in\R^{n\times n}$ be a matrix with positive entries. 
\begin{ex}Define $F\colon\R^n_+\to\R^n_+$ as $F(\bx)=M\bx$. Then, we have $F\in\NB^1$ with $\A(F)=1$ and the eigenvectors of $F$ are the nonnegative eigenvectors of $M$. 
\end{ex}
\begin{ex}
Define $G\colon\R^{m}_+\times \R_+^n\to\R_+^m\times \R^n_+$ as $G(\bx,\by)=(M^{\top}\by,M\bx)$. Then, we have $G\in\NB^2$ with \begin{equation}\label{sd13sd231}\A(G)=\begin{pmatrix}0 & 1 \\ 1 & 0 \end{pmatrix}\end{equation} and the eigenvectors of $G$ in $\S= \{(\bx,\by)\mid \norm{\bx}_2=\norm{\by}_q=2\}$ are the nonnegative singular vectors of $M$.
\end{ex}
\begin{ex}
Let $G\in \NB^2$ be as in the previous example and define $H\colon\R^{m}_+\times \R_+^n\to\R_+^m\times \R^n_+$ as 
$H(\bx,\by)=\big((M^{\top}\by)^{1/(p-1)},(M\bx)^{1/(q-1)})$ where the powers are taken component-wise and $p,q>1$. Then, $H\in\NB^2$ with $\A(H)=\begin{psmallmatrix}\frac{1}{p-1} & 0\\ 0 & \frac{1}{q-1}\end{psmallmatrix}\A(G)$, where $\A(G)$ is as in \eqref{sd13sd231} and the eigenvectors of $H$ in $\S= \{(\bx,\by)\mid \norm{\bx}_p=\norm{\by}_q=1\}$ are the so-called nonnegative $\ell^{p,q}$-singulars vectors of $M$ \cite{Boyd}.
\end{ex}

\subsection{Notation}\label{notation}
In this paper we use the following notation: We use bold letters without index to denote  elements of $V$, bold letters with index $i\in [d]$ denote vectors in $\R^{n_i}$, whereas components of $\bx_i$ are written in normal font. Namely \begin{equation*}
\bx = (\bx_1,\ldots,\bx_d)\in V, \qquad \bx_i=(x_{i,1},\ldots,x_{i,n_i})\in \R^{n_i} \andd x_{i,j_i} \in \R.\end{equation*}

A similar notation is used for mappings $F\colon V\to V$. Namely we let $F=(F_1,\ldots,F_d)$ and $F_i=(F_{i,1},\ldots,F_{i,n_i})$ with $F_i\colon V\to\R^{n_i}$ and $F_{i,j_i}\colon V\to\R$.
Moreover, in order to index the entries of the vectors in $V$ in a more compact way, we consider the following sets of indices
\begin{equation*}
\I = \cup_{i=1}^d\{i\}\times [n_i], \qquad \J = [n_1]\times [n_2]\times\ldots \times [n_d].
\end{equation*}

We will assume each  $\R^{n_i}$ to be equipped with a norm $\|\cdot\|_{\gamma_i}$. For simplicity, we will always assume that the norms $\norm{\cdot}_{\gamma_1},\ldots,\norm{\cdot}_{\gamma_d}$ are monotonic, meaning that $\norm{\bx_i}_{\gamma_i} \leq \norm{\by_i}_{\gamma_i}$ whenever $|\bx_i|\leq |\by_i|$, where the absolute value is taken entrywise. For example, $\|\cdot \|_{\gamma_i}$ can be the Euclidean or any $\ell^p$ norm. Note that here, and in the rest of the paper, $|\bx|$ denotes the componentwise absolute value of $\bx$.  

Given $d$ such norms, we consider the following unit sphere on $V$
\begin{equation*}
\Sn=\big\{\bxb\in V \, : \,   \norm{\bx_i}_{\gamma_i}=1,\,  \forall i \in[d]\big\}, 
\end{equation*}
and we let $\Sn_+ = \Sn\cap \kone_+$ and $\Sn_{++}= \Sn \cap\kone_{++}$.

We will often use the two mappings $\mathfrak M, \mathfrak m\colon\kone_{++}\times\kone_{++}\to\R^d_{++}$,  defined as 
\begin{align*}
\Bmax{\bxb}{\byb}&=\big(\bmaxi{1}{\bxb}{\byb},\ldots,\bmaxi{d}{\bxb}{\byb}\big) =\bigg(\max_{j_1\in[n_1]}\frac{x_{1,j_1}}{y_{1,j_1}} ,\ldots,\max_{j_d\in[n_d]}\frac{x_{d,j_d}}{y_{d,j_d}}\bigg),\\
\Bmin{\bxb}{\byb}&=\big(\bmini{1}{\bxb}{\byb},\ldots,\bmini{d}{\bxb}{\byb}\big) =\bigg(\min_{j_1\in[n_1]}\frac{x_{1,j_1}}{y_{1,j_1}} ,\ldots,\min_{j_d\in[n_d]}\frac{x_{d,j_d}}{y_{d,j_d}}\bigg),
\end{align*}
for every $\bxb,\byb\in\kone_{++}$. Note that, as $\kone_{+,0}$ is closed,  we have 
\begin{equation*}
\Bmin{\bxb}{\byb}\krog\byb\ \lek \ \bxb \ \lek \ \Bmax{\bxb}{\byb}\krog\byb \qquad \forall \bxb,\byb\in\kone_{++}.
\end{equation*}

With $\mathfrak M$ and $\mathfrak m$ we can define two important tools we will use often in our results: the weighted Hilbert and Thompson's metrics on $\kone_{++}$. 
\begin{defi}
Let $\bb\in\R^d_{++}$ be such that $\sum_i b_i=1$. The weighted Hilbert metric $\mu_{\bb}\colon\kone_{++}\times\kone_{++}\to\R_{+}$ and the weighted Thompson metric $\bar\mu_{\bb}\colon\kone_{++}\times\kone_{++}\to\R_{+}$ are defined as
\begin{align*}
\mu_{\bb}(\bxb,\byb)&=\sum_{i=1}^db_i\ln\!\bigg(\frac{\bmaxi{i}{\bx}{\by}}{\bmini{i}{\bx}{\by}}\bigg)=\ln\!\bigg(\prod_{i=1}^d\frac{\bmaxi{i}{\bx}{\by}^{b_i}}{\bmini{i}{\bx}{\by}^{b_i}}\bigg),\\
\bar\mu_{\bb}(\bxb,\byb)&=\sum_{i=1}^db_i\ln\!\Big(\max\big\{\bmaxi{i}{\bx}{\by},\bmaxi{i}{\by}{\bx}\big\}\Big).
\end{align*}
\end{defi}

Note that, in particular,  it follows from Corollary 2.5.6 in \cite{NB} that $(\Sn_{++},\mu_{\bb})$ and $(\kone_{++},\bar\mu_{\bb})$ are complete metric spaces and their topology coincides with the norm topology, for any choice of the positive weights $\bb$.

\subsection{Preliminary properties and results}
The matrix-power operation \eqref{eq:a^A} has some useful algebraic properties which can be proved with a direct computation and that we summarize below. For every $\bal,\bbe\in\R^n_{++}$ and every nonnegative matrices $B,C\in\R^{n\times n}$, we have
\begin{equation}\label{homoident}
\bal^{B}\circ \bal^{C}=\bal^{B+C}, \qquad \big(\bal^{C}\big)^B=\bal^{BC} \qquad \text{and}\qquad \big(\bal\circ\bbe\big)^B=\bal^B\circ\bbe^B,
\end{equation}
where $\circ$ denotes the entrywise product, i.e. $\bal\circ \bbe =(\alpha_1\beta_1,\ldots,\alpha_n\beta_n)$. Moreover, if $\ba\in\R_{++}^n$ and $\lambda\in\R_{++}$, then 
\begin{equation*}
\prod_{i=1}^n\big(\bal^B\big)_i^{a_i}=\prod_{i=1}^n\alpha_i^{(B^\top\ba)_i}\quad\text{ and }\quad (\lambda^{a_1},\ldots,\lambda^{a_n})^B =(\lambda^{(B\ba)_1},\ldots,\lambda^{(B\ba)_n}).
\end{equation*}

Exploiting this formulas, one can easily verify  that the class $\NB^d$ is closed under several natural operations. We list some of  them in the following
\begin{lem}\label{homoNBd}
Let $F,G\in\NB^d$, $A=\A(F)$ and $B=\A(G)$. Moreover, let $D\in\R^{d\times d}$ with $D\geq A,B$ and for $i\in[d]$ let $\xi_i\colon\R^{n_i}_{+}\to\R_+$ be continuous, order-preserving, $1$-homogeneous mappings such that $\xi_i(\R^{n_i}_+\saufzero)\subset\R_{++}$. Define $N\colon\kone_+\to\R^d_+$ as $N(\bx)=\big(\xi_1(\bx),\ldots,\xi_d(\bx)\big)$. Finally, let $H^{(1)},H^{(2)},H^{(3)}\colon\kone_{+}\to\kone_+$ with
\begin{align*}
H^{(1)}(\bx)&= F\big(G(\bx)\big), \qquad H^{(2)}(\bx)=F(\bx)\circ G(\bx),\\
H^{(3)}(\bx)&=N(\bx)^{D-A}\krog F(\bx)+N(\bx)^{D-B}\krog G(\bx),
\end{align*}
where, in the definition of $H^{(2)}$, $\circ$ denotes the entrywise product. 

Then $H^{(1)},H^{(2)},H^{(3)}\in\NB^d$ with homogeneity matrices $AB$, $A+B$, $D$, respectively. In particular, for every $F\in\NB^d$ we have $\A(F^k)=\A(F)^k$, where $F^k$ denotes $k$ compositions of $F$ with itself. 
\end{lem}

If $F$ is differentiable at $\bx$, we write  $DF(\bx)$ to denote the Jacobian matrix of $F$. We recall below a known theorem that shows that the differential of a mapping $F\in\NB^d$ is order-preserving as well. 
\begin{thm}[Theorem 1.3.1, \cite{NB}]\label{opcharac}
Let $U\subset \kone_+$ be an open convex set. If $F\colon U \to \kone_{+}$ is locally Lipschitz, then $DF(\bx)$ exists for Lebesgue almost all $\bx\in U$, and $F$ is order-preserving if and only if $DF(\bx)\kone_{+}\subset\kone_{+}$ for all $\bx\in U$ for which $DF(\bx)$ exists.
\end{thm}

The next lemma generalizes Euler's theorem for homogeneous mappings to \multihomo{} mappings. It characterizes multi-homogeneous mappings and provides information on the multi-homogeneity of their derivatives. For $U\subseteq V$ and a map $f\colon U \to \R$, denote by $\grad_if(\bx)$ the gradient of $g_i(\by_i)= f(\bx_1,\ldots,\bx_{i-1},\by_i,\bx_{i+1},\ldots,\bx_d)$ at $\by_i=\bx_i$. If there exists $\ba\in\R^d$  such that $f$ satisfies $f(\bal\krog\bx)=f(\bx)\prod^d_{k=1}\alpha_k^{a_k}$ for all $\bx\in U$ and $\bal\in\R^d_{++}$, then $g_i$ is positively homogeneous of degree $a_{i}$ for all $i\in[d]$. With this observation, the following result is a direct consequence of Euler's theorem for homogeneous functions applied to $g_1,\ldots,g_d$ and therefore its proof is omitted.
\begin{lem}\label{Eulerthm}
Let $U\subset V$ be open and such that  $\bal\krog\bx\in U$ for all $\bal\in\R^d_{++}$ and $\bx\in U$. Let $\ba\in\R^d$ and $f\colon U\to\R$, a differentiable mapping. The following are equivalent:
\begin{enumerate}[(1)]
\item It holds\label{homomulti1} $f(\bal\krog\bx)=f(\bx)\prod^d_{k=1}\alpha_k^{a_k}$ for every $\bal\in\R^d_{++},\ \bx\in U.$
\item It holds
\label{homomulti2}
$\ps{\grad_if(\bx)}{\bx_i}=a_if(\bx)$ for every $i \in[d],\ \bx\in U.$
\end{enumerate}
Moreover, if $f$ satisfies \eqref{homomulti1} or \eqref{homomulti2}, then:
\begin{enumerate}[(1)]\setcounter{enumi}{2}
\item It holds $\grad_i f(\bal\krog\bx)=\grad_i f(\bx)\alpha_i^{-1}\prod_{k=1}^d\alpha_k^{a_k} $ for all $i\in[d],\bal\in\R^d_{++},\bx\in U.$
\end{enumerate}
\end{lem}
There exist order-preserving multi-homogeneous mappings which are naturally defined on $\kone_{++}$ rather than on $\kone_+$. This frequently happens in the case $d=1$ when considering the log-exp transform of topical mappings, i.e.\ order-preserving mappings $F\colon \R^{n}\to \R^n$ satisfying $F(\bx+\lambda \ones)=F(\bx)+\lambda \ones$ for all $\bx\in \R^n, \lambda \in \R$ (see e.g. \cite{Gaub_survey} and \cite[Section 1.5]{NB}). We also face such a situation when deriving a dual condition for the existence of a positive eigenvector in Corollary \ref{existdual}. It is then useful to know whether the considered mapping can be continuously extended to a mapping in $\NB^d$. In the case $d=1$, such an extension has been proved to exist in Theorem 3.10 \cite{cont_ext} and Theorem 5.1.2 \cite{NB}. As the proof of this result can be easily generalized for $d>1$ (with the help of Lemma \ref{contract}), we omit it here.
\begin{thm}\label{extend}
Let $F\colon\kone_{++}\to\kone_{++}$ be order-preserving and \multihomo.\!\! If $\A(F)$ has at least one positive entry per row and there exists $\bb\in\R^d_{++}$ such that $\A(F)^\top\bb\leq \bb$, then there exists $\bar F\in\NB^d$ such that $F=\bar F|_{\kone_{++}}$ and $\A(\bar F)=\A(F)$.
\end{thm}

\section{Contraction principle for Multi-homogeneous mappings}\label{contrasec}
Our first result is a combination of \eqref{Lipschitz} with the Banach fixed point theorem. This result is particularly interesting as it shows that when we can build a metric so that $F\in\NB^d$ is a strict contraction then the existence and uniqueness of a positive eigenvector are always guaranteed without further assumptions. As discussed below \eqref{Lipschitz}, such a metric can be explicitly constructed using the left eigenvector of the homogeneity matrix of $F$ in order to obtain the following:
\begin{thm}\label{Banachcombi}
Let $F\in\NB^d$ and $A=\A(F)$. If $\rho(A)<1$, then $F$ has a unique positive eigenvector $\bx\in \kone_{++}$ up to rescaling of $\bx_i$ for $i\in[d]$.
\end{thm}
The proof of this result is postponed to the end of the next Subsection \ref{subsec:thm31}.

The simplicity of the assumptions in the above theorem is remarkable. While this result was known in the case $d=1$ (see for instance \cite{Bus73}), it has strong novel implications in the Perron-Frobenius theory for spectral problems induced by nonnegative tensors, which we discuss in \cite{tensorpaper}. A simple consequence of Theorem \ref{Banachcombi} is the following: Let $M\in\R^{m\times n}$ be a nonnegative matrix, then the nonlinear power method of \cite{Boyd} for the estimation of $\norm{M}_{p,q}=\max\{\norm{M\bx}_{p}\mid \norm{\bx}_{q}=1\}$ always converges to the global maximum, whenever $p<q$ and $M^\top M$ has at least one nonzero entry per row. The existing convergence result for this method requires $M^\top M$ to be irreducible which is much more restrictive. 

Unfortunately, the eigenvalue problem $M\bx = \lambda \bx$ where $M\in\R^{n\times n}$ is a matrix with positive entries and $\bx\in\R^n_+$, does not satisfy the assumptions of Theorem \ref{Banachcombi} because in this particular case, $F(\bx)=M\bx$ is one homogeneous and so $\A(F)=1$. That is, $F$ is nonexpansive but may not be a strict contraction. This explains to some extent why the linear Perron-Frobenius theorem requires $M$ to be irreducible and not simply $F(\R^n_{++})\subset \R^n_{++}$. To distinguish these cases and facilitate our discussion, for a mapping $F\in\NB^d$, we say that $F$ is a (strict) contraction if $\rho(\A(F))<1$ and say that $F$ is nonexpansive if $\rho(\A(F))=1$. As for the case $d=1$, when $d>1$ the study of nonexpansive mappings is more involved than that of strict contractions.

\subsection{Lipschitz continuity and the contraction principle}\label{subsec:thm31}
The following lemma provides a Lipschitz constant for $F\in \NB^d$ with respect to the weighted Hilbert and Thompson metrics.
\begin{lem}\label{contract}
Let $F\in\NB^d$, $A=\A(F)$, $\bb\in\R^d_{++}$. For every $\bx,\by\in\kone_{++}$, it holds
\begin{equation}\label{nnnexpeq}
\mu_{\bb}\big(F(\bx),F(\by)\big)\,\leq\, C\, \mu_{\bb}(\bx,\by) \andd \bar\mu_{\bb}\big(F(\bx),F(\by)\big)\,\leq\, C\, \bar\mu_{\bb}(\bx,\by). 
\end{equation}
where $C= \max\big\{(A^\top\bb)_i/b_i\ \big|\ i\in[d]\big\}$.
\end{lem}
\begin{proof}
For any $\bxb,\byb\in\kone_{++}$, we have
\begin{equation}\label{homoineq}
\Bmin{\bx}{\by}^A\krog F(\by)\ \lek\ F(\bx)\ \lek\ \Bmax{\bx}{\by}^A\krog F(\by).
\end{equation}
It follows that for every $(j_1,\ldots,j_d)\in\J$ it holds
\begin{equation*}
\prod_{i=1}^d\bmini{i}{\bx}{\by}^{(A^\top\bb)_i}\leq \prod_{i=1}^d\bigg(\frac{F_{i,j_i}(\bx)}{F_{i,j_i}(\by)}\bigg)^{b_i}\leq \prod_{i=1}^d\bmaxi{i}{\bx}{\by}^{(A^\top\bb)_i}.
\end{equation*}
Hence, we have
\begin{align*}
\mu_{\bb}\big(F(\bxb),F(\byb)\big)&= \sum_{i=1}^db_i\ln\!\bigg(\frac{\bmaxi{i}{F(\bx)}{F(\byb)}}{\bmini{i}{F(\bxb)}{F(\byb)}}\bigg)\leq \sum_{i=1}^d(A^\top\bb)_i\ln\!\bigg(\frac{\bmaxi{i}{\bxb}{\byb}}{\bmini{i}{\bxb}{\byb}}\bigg)\\
&=  \sum_{i=1}^d\frac{(A^\top\bb)_i}{b_i}b_i\ln\!\bigg(\frac{\bmaxi{i}{\bxb}{\byb}}{\bmini{i}{\bxb}{\byb}}\bigg)\leq C\,\mu_{\bb}(\bxb,\byb).
\end{align*}
Furthermore, Equation \eqref{homoineq} implies that
\begin{align*}
\bar\mu_{\bb}\big(F(\bxb),F(\byb)\big)
&\leq \ln\!\Bigg(\prod_{i=1}^d\max\!\bigg\{\prod_{k=1}^d\bmaxi{k}{\bx}{\by}^{A_{i,k}},\prod_{k=1}^d\bmaxi{k}{\by}{\bx}^{A_{i,k}}\bigg\}^{b_i}\Bigg)\\
&\leq \ln\!\bigg(\prod_{k=1}^d\max\!\Big\{\bmaxi{k}{\bx}{\by},\bmaxi{k}{\by}{\bx}\Big\}^{(A^\top\bb)_k}\bigg)\leq C\,\bar\mu_{\bb}\big(\bxb,\byb\big),
\end{align*}
which concludes the proof.
\end{proof}
The constant $C$ in the above lemma cannot be improved further without additional assumptions on $F\in\NB^d$. This fact is illustrated by the following example where we show that for any matrix $A\in\R^d_{+}$ with $A\R^d_{++}\subset \R^d_{++}$, there exists a mapping $F\in\NB^d$ such that $\A(F)=A$ and we have equality in \eqref{nnnexpeq} for some $\bx,\by\in\kone_{++}$ with $\bx\neq \by$. Moreover, the example shows that there exist  mappings $F\in \NB^d$ such that $C\leq 1$ in the lemma above,  even though $F^k$ is expansive with respect to the  Hilbert and Thompson metrics on  $\R^{n_1+\ldots+n_d}_{++}$, for all $k\geq 1$. 
This example motivates the study of multi-homogeneous mappings and illustrates that several arguments involving standard homogeneous mappings do not hold anymore in the multi-homogeneous framework.
\newcommand{\Lip}{\operatorname{Lip}}
\begin{ex}\label{tightmot}
	Let $d\geq 2$, $n\geq 2$, $A\in\R^{d\times d}$ be any matrix such that $A\R^{d}_{++} \subset \R^d_{++}$. Set $n_1=\ldots=n_d=n$ and let $\pi\colon\{1,\ldots,n\}\to \{1,\ldots,n\}$ be a permutation. Define $F\colon\kone_+\to\kone_+$ as $F_{i,j_i}(\bx) = \prod_{l=1}^dx_{l,\pi(j_i)}^{A_{i,l}}$ for every $(i,j_i)\in\I$. Then, we have $F\in\NB^d$ with $\A(F)=A$ and, for every $\bx,\by\in\kone_{++}$,  $\bb\in\R^d_{++}$ and  $(j_1,\ldots,j_d)\in\J$, it holds
	\begin{equation*}
	\prod_{i=1}^d \Big(\frac{F_{i,j_i}(\bx)}{F_{i,j_i}(\by)}\Big)^{b_i}=\prod_{i=1}^d \Big(\prod_{l=1}^d\Big(\frac{x_{l,\pi(j_i)}}{y_{l,\pi(j_i)}}\Big)^{A_{i,l}}\Big)^{b_i}=\prod_{i=1}^d\Big(\frac{x_{i,\pi(j_i)}}{y_{i,\pi(j_i)}}\Big)^{(A^\top\bb)_i}.
	\end{equation*}
	Hence, for all $\bx,\by\in\kone_{++}$, we have
	\begin{equation}\label{coolrelation}
	\mu_{\bb}(F(\bx),F(\by)) = \mu_{A^\top \bb}(\bx,\by) \andd \bar\mu_{\bb}(F(\bx),F(\by)) = \bar\mu_{A^\top \bb}(\bx,\by).
	\end{equation}
	It follows that for all $i\in[d]$, if $\bx_j=\by_j$ for all $j\in[d]\sauf\{i\}$ and $\mu_1(\bx_i,\by_i)>0$, where $\mu_1\colon\R^n_{++}\times \R^n_{++}\to \R_+$ denotes the Hilbert metric \eqref{eq:standard_H_metric} on $\R^n_{++}$, then 
	$$\mu_{\bb}(F(\bx),F(\by)) = \frac{(A^\top\bb)_i}{b_i}\mu_{\bb}(\bx,\by)\andd \bar\mu_{\bb}(F(\bx),F(\by)) = \frac{(A^\top\bb)_i}{b_i}\bar\mu_{\bb}(\bx,\by),$$
	and thus, there exists $\bx,\by\in \kone_{++}$ such that $\mu_{\bb}(F(\bx),F(\by))=C_{\bb}\mu_{\bb}(\bx,\by)$ and
	$\bar\mu_{\bb}(F(\bx),F(\by))=C_{\bb} \bar\mu_{\bb}(\bx,\by)$ where $C_{\bb}=\max_{i\in[d]}(A^\top\bb)_i/b_i$.

	We now show that the matrix $A$ can be chosen so that $F$ is nonnexpansive with respect to $\mu_{\bb}$ and $\bar\mu_{\bb}$, whereas $F^m$ is expansive with respect to the standard Hilbert and Thompson's metrics $\mu$ and $\bar\mu$ on the ``flattened'' space $\R^{dn}_{+}$, for all $m\geq 1$.

	 For convenience, for  $\delta\colon\kone_{++}\times \kone_{++}\to \R_+$ and  $G\colon \kone_{++}\to \kone_{++}$, let us denote the smallest Lipschitz constant of $G$ with respect to $\delta$ as
	$$\Lip(G,\delta)=\inf\big\{c>0\ \big|\ \delta(G(\bx),G(\by))\leq c\,\delta(\bx,\by),\, \forall \bx,\by\in\kone_{++}\big\}.$$
	Then, the above discussion together with Lemma \ref{contract} imply that 
	\begin{align}\label{exactLIP}
	\max_{i\in[d]}\frac{(A^\top\bb)_i}{b_i}=\Lip(F,\mu_{\bb})=\Lip(F,\bar\mu_{\bb}).
	\end{align}
	Now, 
	it follows from Lemma \ref{homoNBd} that for every $m\geq 1$, it holds
	$F^m_{i,j_i}(\bx) = \prod_{l=1}^dx_{l,\pi^m(j_i)}^{(A^m)_{i,l}}$ for every $(i,j_i)\in\I$. Therefore, if $\bb$ is a positive eigenvector such that $A^\top \bb = \rho(A)\bb$, then by \eqref{coolrelation} and \eqref{exactLIP} we have 
	$
	\rho(A)^m =\Lip(F^m,\mu_{\bb})=\Lip(F^m,\bar\mu_{\bb})$ for all $m\geq 1.$	
	We now show that  for all $m\geq 1$, it holds
	\begin{equation}\label{boundonc}
	 \norm{A^m}_{\infty}\leq\Lip(F^m,\mu)\andd \norm{A^m}_{\infty}\leq\Lip(F^m,\bar\mu).
	\end{equation}
	where, $\norm{\cdot}_{\infty}$ denotes the matrix infinity norm $\norm{M}_{\infty}=\max_{i\in[d]} \sum_{j=1}^d|M_{i,j}|$.
	We prove the claim for $m=1$, as the case $m> 1$ can be easily deduced by substituting $A$ with $A^m$ and $\pi$ with $\pi^m$ in the following argument. 

	For  $s,t>0$, define $\bv^{(s,t)}=(s,t,1,\ldots,1)\in\R^n_{++}$ and $\bx^{(s,t)}=(\bv^{(s,t)},\ldots,\bv^{(s,t)})\in\kone_{++}$.
	Then, for every $s,t,\tilde s,\tilde t>0$, it holds 
	$$ \frac{x_{i,j_i}^{(\tilde s,\tilde t)}}{ x_{i,j_i}^{(s,t)}}=\begin{cases}\tilde s/ s & \text{if }j_i = 1,\\ {\tilde t}/{t} & \text{if }j_i = 2,\\ 1& \text{otherwise},\end{cases}\quad\text{and}\quad \frac{F_{i,j_i}(\bx^{(\tilde s,\tilde t)})}{F_{i,j_i}(\bx^{(s,t)})}=\begin{cases}  \big({\tilde s}/{s}\big)^{(\sum_{l=1}^d A_{i,l})} & \text{if }\pi(j_i) = 1,\\\big({\tilde t}/{t}\big)^{(\sum_{l=1}^d A_{i,l})} & \text{if }\pi(j_i) = 2,\\ 1& \text{otherwise}.\end{cases}$$
	Therefore
	$\mu(F(\bx^{(s,1)}),F(\bx^{(1,t)})) = \norm{A}_{\infty}\mu(\bx^{(s,1)},\bx^{(1,t)})$ and $\bar\mu(F(\bx^{(s,1)}),F(\bx^{(1,t)}))  = \norm{A}_{\infty}\bar\mu(\bx^{(s,1)},\bx^{(1,t)})$,  for every $s>t>1$. Thus \eqref{boundonc} holds as claimed.

	Finally, let us consider the following $3\times 3$ example matrix
	$$A = \frac{1}{4}\begin{pmatrix} 0& 12 & 0\\ 1 & 0 & 1\\ 0 & 4 & 0 \end{pmatrix}\, .$$
	Then, for any integer $m\geq 1$ it holds
	$A^{2m-1}=A$ and 
	$A^{2m} = \frac{1}{4}\begin{psmallmatrix} 3& 0 & 3\\ 0 & 4 & 0\\ 1 & 0 & 1 \end{psmallmatrix}$. 
	Furthermore, $\rho(A)=1$, $\bb =(1,4,1)^\top$ is a positive eigenvector such that $A^\top\bb = \rho(A)\bb$, and for all $m\geq 1$ it holds $\norm{A^{2m-1}}_{\infty}=3$ and $\norm{A^{2m}}_{\infty}=3/2$. Hence, with this particular $A$, we have
	$$\Lip(F^m,\mu_{\bb})=\Lip(F^m,\bar\mu_{\bb})=1<\frac{3}{2}\leq \min\{\Lip(F^m,\mu),\Lip(F^m,\bar\mu)\},$$
	for all $m\geq 1$, i.e. $F^m$ is nonexpansive with respect to the weighted Hilbert and Thompson metrics $\mu_{\bb},\bar\mu_{\bb}$ on $\R^{n}_+\times \R^{n}_+\times \R^{n}_+$ whereas every power of $F$ is expansive with respect to the standard Hilbert and Thompson metrics $\mu,\bar \mu$ on $\R^{3n}_+$.
\end{ex}

We deduce Theorem \ref{Banachcombi} from Lemma \ref{contract}.
\begin{proof}[Proof of Theorem \ref{Banachcombi}]
As $\rho(A)<1$ by assumption, the Collatz-Wielandt principle implies the existence of $\bb\in\R^d_{++}$ such that $C = \max_{i\in[d]}(A^\top \bb )_i/b_i <1$. By Lemma \ref{contract}, we have 
$\mu_{\bb}\big(F(\bxb),F(\byb)\big)\leq C \mu_{\bb}(\bxb,\byb)$ for all $ \bxb,\byb\in\kone_{++}.$
Now, consider the mapping $G\colon\Sn_{++}\to\Sn_{++}$ defined as
$G(\bxb)=\big(\norm{F_1(\bx)}_{\gamma_1}^{-1},\ldots,\norm{F_d(\bx)}_{\gamma_d}^{-1}\big)\krog F(\bx)$, then we have $\mu_{\bb}\big(G(\bx),G(\by)\big)=\mu_{\bb}\big(F(\bx),F(\by)\big)$ for every $\bx,\by\in\Sn_{++}$. Thus,
$G$ is a strict contraction on the complete metric space $(\Sn_{++},\mu_{\bb})$. The result is now a consequence of the Banach fixed point theorem (see e.g. Theorem 3.1 in \cite{pointfixe}).
\end{proof}

\section{ Spectral radius }\label{existence_section}
Maximality plays an important role in the Perron-Frobenius theory. For example, if the eigenvectors of a mapping $F\in\NB^d$ are the critical points of some potential $f\colon\kone_+\to\R$, then we want to assert that nonnegative or positive eigenvectors coincides with the global maximizer of $f$, constrained on some product of unit balls. In this setting, the function $f$ can be regarded as the numerator of a Rayleigh quotient. In order to keep such connections, we propose the following way to compare the ``spectral magnitude'' of eigenvectors. The main idea is to fix the scaling of eigenvectors by imposing unit norm constraints on $\bx_i$ and then take the weighted geometric mean of the eigenvalues $\lambda_1,\ldots,\lambda_d$ associated to $\bx\in\kone_{+,0}$. In particular, note that the eigenvectors of $F\in\NB^d$ can always be rescaled so that they belong to $\Sn_+$. So, for $\bb\in\R^d_{++}$, we introduce the following notion of spectral radius of $F\in\NB^d$
\begin{equation*}
r_{\bb}(F)=\sup \Big\{\prod_{i=1}^d\lambda_i^{b_i}\ \Big|\ F(\bx)=\blam\krog\bx \text{ for some }\bx\in\Sn_+\Big\}.
\end{equation*}
Note that $r_{\bb}(F)$ is always nonnegative, as $F(\kone_+)\subset\kone_+$ and so $F(\bx)=\blam\krog\bx$ implies $\lambda_i\geq0$ for all $i\in[d]$.
By Theorem \ref{Banachcombi}, it is clear that $r_{\bb}(F)$ is well defined for strict contractions in $\NB^d$. It is however less clear that, in the case where $F$ is nonexpansive, the supremum above is not taken over an empty set. This issue is addressed by the next theorem which can be seen as a generalization of the weak Perron-Frobenius theorem. In particular, it is shown that every nonexpansive mapping $F\in\NB^d$ for which there exists $\bb\in\R^d_{++}$ with $\A(F)^\top\bb=\bb$, has a nonnegative eigenvector with eigenvalue corresponding to $r_{\bb}(F)$. A proof for the case $d=1$ can be found in Theorem 5.4.1 \cite{NB} and essentially relies on the fact that the spectral radius of an order-preserving homogeneous mapping can be characterized in terms of its Bonsall spectral radius \cite{Bonsall} and in terms of its cone spectral radius \cite{cone_spec_rad}. By generalizing these characterizations, we obtain the following:
\begin{thm}\label{weak_PF}
Let $F\in\NB^d$ and $A=\A(F)$. If there exists  $\bb\in\R^d_{++}$ such that $\sum_ib_i=1$ and $A^\top\bb=\bb$,  then there exists $\bu\in\Sn_+$ and $\blam\in\R^d_+$ such that $F(\bu)=\blam\krog\bu$ and $r_{\bb}(F)=\prod_{i=1}^d\lambda_i^{b_i}$. Furthermore, it holds
\begin{equation*}\label{specrad_charac}
r_{\bb}(F)=\sup_{\bx\in\kone_{+,0}} \limsup_{m\to\infty}\prod_{i=1}^d\norm{F^m_i(\bx)}_{\gamma_i}^{b_i/m}=\lim_{m\to\infty}\sup_{\bx\in\Sn_+}\prod_{i=1}^d\norm{F^m_i(\bx)}_{\gamma_i}^{b_i/m}.
\end{equation*}
\end{thm}
The proof of this theorem relies on a number of preliminary results and thus is postponed to the end of the next Subsection \ref{specrad_sec}

\subsection{Spectral radius of nonexpansive mappings}\label{specrad_sec}
\newcommand{\abs}[2]{\langle #1 \rangle_{#2} ZZ}
We consider the notions of Bonsall spectral radius and cone spectral radius for mappings $F\in\NB^d$ such that there exists $\bb>0$ with $\A(F)^\top\bb=\bb$. This allows us to show that the supremum in the definition of $r_{\bb}(F)$ is attained. 

For convenience, from now on let us denote by $\Dn$ the set
$$
\Dn = \{\bz \in \R^d_{++} : z_1+\dots+z_d = 1\}\, .
$$

Let $\bx\in\kone_+$, $F\in\NB^d$, $A=\A(F)$, $\bb\in\Dn$ and assume that $A^\top\bb =\bb$. Define
\begin{equation*}
\nnorm{\bxb}_{\bb}= \prod_{i=1}^d\norm{\bx_i}_{\gamma_i}^{b_i}\andd \nnorm{F}_{\bb}= \sup_{\bxb \in \Sn_+}\nnorm{F(\bxb)}_{\bb}.
\end{equation*}
Then, for every $\bal\in\R^d_{++}$ and $\bx \in\kone_{+,0}$, it holds
\begin{equation*}
\nnorm{F(\bal\krog \bx)}_{\bb} = \nnorm{\bal^A\krog F(\bx)}_{\bb} =\nnorm{F(\bx)}_{\bb}\prod_{i=1}^d\alpha_i^{(A^\top\bb)_i}=\nnorm{F(\bx)}_{\bb}\prod_{i=1}^d\alpha_i^{b_i}
\end{equation*}
Hence, with $\bbe= (\norm{\bx_1}^{-1}_{\gamma_1},\ldots,\norm{\bx_d}^{-1}_{\gamma_d})$, we have
\begin{equation}\label{opnormineq}
\nnorm{F(\bxb)}_{\bb}=\nnorm{F(\bbe\krog\bxb)}_{\bb}\nnorm{\bxb}_{\bb}\leq\nnorm{F}_{\bb}\nnorm{\bxb}_{\bb} \qquad \forall \bxb\in\kone_{+,0}.
\end{equation}
Now, consider
\begin{equation*}
\bar r_{\bb}(F)= \sup_{\bxb\in\kone_{+,0}}\limsup_{m\to \infty}\nnorm{F^{m}(\bxb)}_{\bb}^{1/m}
\andd
\hat{r}_{\bb}(F)= \lim_{m\to\infty}\nnorm{F^{m}}_{\bb}^{1/m}.
\end{equation*} 
In the case $d=1$, $\hat r_{\bb}$ is known as Bonsall spectral radius \cite{Bonsall} and $\bar r_{\bb}$ is known as cone spectral radius \cite{cone_spec_rad}. Note that for every $\lambda >0$, as $\sum_{i=1}^d b_i=1$, we have
\begin{equation*}
\bar r_{\bb}(\lambda\, F)= \lambda\, \bar r_{\bb}(F)\andd \hat r_{\bb}(\lambda\, F)= \lambda \,\hat r_{\bb}(F).
\end{equation*}
Moreover, if $M\in\R^{n\times n}$ is a nonnegative matrix and $F(\bx)=M\bx$, then the Gelfand formula \cite{Gelfand} implies that $\rho(M)=\bar r_{1}(F)=\hat r_{1}(F)$. The proof of Theorem 5.3.1 \cite{NB}, a special case of Theorem 2.2 \cite{cone_spec_rad}, can be easily adapted to obtain the following theorem, whose proof is omitted for brevity.
\begin{thm}\label{Thm531} 
Let $F\in\NB^d$, $A=\A(F)$ and $\bb\in\Dn$ with $A^\top\bb=\bb$, then it holds
$0 \leq  \bar r_{\bb}(F) = \hat{r}_{\bb}(F) <\infty.$
\end{thm}
In the following proposition we extend the second part of Theorem 2.2 \cite{cone_spec_rad} to the multi-homogeneous case. In particular, it implies that if $F\in\NB^d$ is nonexpansive and has a positive eigenvector $\bu\in\Sn_{++}$ with $F(\bu)=\blam \krog\bu$, then $\bar r_{\bb}(F)=\prod_{i=1}^d\lambda_i^{b_i}$. Moreover, we use this proposition for the proof of the Collatz-Wielandt formula in Section \ref{CW_M_U1}.
\begin{prop}\label{Prop536}
Let $F\in\NB^d$, $A=\A(F)$ and $\bb\in\Dn$ with $A^\top\bb=\bb$. Then,
$\bar r_{\bb}(F)=\lim_{m\to\infty}\nnorm{F^{m}(\bxb)}_{\bb}^{1/m}$ for all $ \bxb\in\kone_{++}$.
Moreover, for every $\by\in \Sn_{+}$ and $\bt\in\R^d_{++}$ with $\bt\krog \by \lek F(\by)$, we have 
$\prod_{i=1}^d\theta_i^{b_i}\leq \bar r_{\bb}(F).$
\end{prop}
\begin{proof}
Let $\bxb\in\kone_{++}$, there exists $\bs\in\R^d_{++}$ such that for every $\byb \in \Sn_+$, it holds $\byb \lek \bs\krog\bxb$. Let $\sigma = \prod_{i=1}^d s_i^{b_i}$. For $k \in \N$ and $\byb \in  \Sn_+$ we have 
\begin{equation*}
\nnorm{F^{k}(\byb)}_{\bb}\leq \nnorm{F^{k}(\bs\krog\bxb)}_{\bb}= \nnorm{\bs^{A^k}\krog F^{k}(\bxb)}_{\bb}=\nnorm{F^{k}(\bxb)}_{\bb}\sigma. 
\end{equation*}
It follows that
\begin{equation*}
\nnorm{F^{k}}_{\bb}=\sup_{\by\in\Sn_+}\nnorm{F^{k}(\byb)}_{\bb}\leq\nnorm{F^{k}(\bxb)}_{\bb}\sigma, 
\end{equation*}
so that, with $(A^k)^\top\bb = \bb$ and \eqref{opnormineq}, we get
\begin{equation*}
\nnorm{F^{k}}_{\bb}\sigma^{-1}\leq \nnorm{F^{k}(\bxb)}_{\bb} \leq \nnorm{F^{k}}_{\bb}\nnorm{\bxb}_{\bb} \qquad \forall k\in\N.
\end{equation*}
Theorem \ref{Thm531} implies $\lim_{k\to \infty}\nnorm{F^{k}}^{1/k}_{\bb}=\bar r_{\bb}(F)$, hence
$\bar r_{\bb}(F)=\lim_{k\to\infty}\nnorm{F^k(\bxb)}^{1/k}_{\bb}.$
Now, if $\bt\krog\by\lek F(\by)$, then for all $k\in\N$
\begin{equation*}
\nnorm{F^{k}(\by)}_{\bb}\geq\nnorm{\by}_{\bb} \prod_{i=1}^d\theta_i^{(\sum_{j=0}^{k-1} (A^j)^\top\bb)_i} =\nnorm{\by}_{\bb}\prod_{i=1}^d\theta_i^{kb_i} .
\end{equation*}
Hence,
\begin{equation*}
\prod_{i=1}^d\theta_i^{b_i} = \lim_{k\to \infty}\Big(\nnorm{\by}_{\bb}\prod_{i=1}^d\theta_i^{kb_i}\Big)^{1/k} \leq \limsup_{k\to \infty}\nnorm{F^{k}(\by)}_{\bb}^{1/k} \leq \bar r_{\bb}(F),
\end{equation*}
which concludes the proof.
\end{proof}
The last tool we need to prove the weak Perron-Frobenius Theorem \ref{weak_PF}, is the next result which is a generalization of Theorem 5.4.1 \cite{NB}, where we prove the existence of an eigenvector corresponding to the spectral radius for a class of mappings in $\NB^d$. Although being of interest in its own, this theorem will also be helpful in Section \ref{wirr_thm} for the proof of the existence of a positive eigenvector. Furthermore, we will use it in Section \ref{CW_M_U} to show that the Collatz-Wielandt characterization of the spectral radius holds without the assumption that there exists a positive eigenvector.
\begin{thm}\label{BIGTHM}
Let $F\in\NB^d$, $A=\A(F)$ and $\bb\in\Dn$ with $A^\top\bb=\bb$. For each $\delta>0$, define $F^{(\delta)}\colon\kone_{+}\to \kone_{+}$ as
\begin{equation*}
F^{(\delta)}(\bx)= F(\bx)+\delta\big(\norm{\bx_1}_{\gamma_1},\ldots,\norm{\bx_d}_{\gamma_d}\big)^A\krog\ones,
\end{equation*}
where $\ones$ is the vector of all ones. Then, the following statements hold:
\begin{enumerate}[(1)]
\item For every $\delta>0$, we have $F^{(\delta)}\in\NB^d$, $\A(F^{(\delta)})=A$ and there exists $(\blam^{(\delta)},\bxb^{(\delta)})$ in $\R^d_{++}\times\Sn_{++}$ such that
\label{abra2} 
$F^{(\delta)}(\bxb^{(\delta)})=\blam^{(\delta)}\krog\bxb^{(\delta)}$ and 
$ \prod_{i=1}^d (\lambda_i^{(\delta)})^{b_i}=\bar r_{\bb}(F^{(\delta)}).$
\item If $0<\eta<\epsilon$, then $\bar r_{\bb}(F^{(\eta)})<\bar r_{\bb}(F^{(\epsilon)})$ and hence $\lim_{\delta\to 0}\bar r_{\bb}(F^{(\delta)}) = r$ exists.\label{abra3}
\item There exists $(F^{(\delta_k)})_{k=1}^{\infty}\subset \{F^{(\delta)}\}_{\delta>0}$ such that $\lim_{k\to\infty}F^{(\delta_k)}= F$ and the corresponding sequence $(\blam^{(\delta_k)},\bx^{(\delta_k)})_{k=1}^{\infty}$ obtained from \eqref{abra2}, converges to a maximal eigenpair of $F$. That is, there exists a pair $(\blam,\bx)\in\R^d_{+}\times \S_+$ such that\label{abra4} it holds $\lim_{k\to\infty}(\blam^{(\delta_k)},\bxb^{(\delta_k)}) =(\blam,\bx)$, 
$F(\bx)=\blam \krog \bx$ and $\bar r_{\bb}(F)=\prod_{i=1}^d \lambda_i^{b_i}=r.$
\end{enumerate}
\end{thm}
\begin{proof}
We prove \eqref{abra2}:
Let $\delta>0$, then $F^{(\delta)}\in\NB^d$ and $A=\A(F^{(\delta)})$ follow from Lemma \ref{homoNBd}. 
Let $\bphi\in\kone_{++}$ and let $ \S^{\bphi}_{+}=\big\{\bx\in\kone_+ \ \big|\ \ps{\bx_i}{\bphi_i}=1, \ \forall i\big\}$.  Since $F^{(\delta)}(\kone_{+,0})\subset \kone_{++}$, the mapping $G^{(\delta)}\colon\S^{\bphi}_{+}\to\S^{\bphi}_+$, defined by
\begin{equation*}
G^{(\delta)}(\bz) =\big(\langle{\bphi_1},{F_1^{(\delta)}(\bz)}\rangle^{-1},\ldots,\langle{\bphi_d},{F_d^{(\delta)}(\bz)}\rangle^{-1}\big)\krog F^{(\delta)}(\bz).
\end{equation*} 
is well defined and continuous. It follows from the Brouwer fixed point theorem (see for instance \cite{Brouwer}) that $G^{(\delta)}$ has a fixed point $\tilde\bx^{(\delta)}\in\S^{\bphi}_+$. We have $\tilde \bx^{(\delta)} \in\kone_{++}$ as $F^{(\delta)}(\kone_{+,0})\subset \kone_{++}$. By rescaling $\tilde \bx^{(\delta)}$ and using $G^{(\delta)}(\tilde \bx^{(\delta)})=\tilde \bx^{(\delta)}$ we obtain the existence of $(\blam^{(\delta)},\bx^{(\delta)})\in\R^d_{++}\times \Sn_{++}$ such that $F^{(\delta)}(\bxb^{(\delta)})=\blam^{(\delta)}\krog\bxb^{(\delta)}$. By Proposition \ref{Prop536}, we know $\prod_{i=1}^d (\lambda_i^{(\delta)})^{b_i}=\bar r_{\bb}(F^{(\delta)})$. \\
We prove \eqref{abra3}:
Let $0<\eta<\epsilon$. As $F^{(\eta)}(\bxb^{(\eta)})=\blam^{(\eta)}\krog\bxb^{(\eta)}$, we have 
\begin{equation*}
F^{(\epsilon)}(\bxb^{(\eta)})=F(\bx^{(\eta)})+\epsilon\ones=F^{(\eta)}(\bxb^{(\eta)})+(\epsilon-\eta)\ones=\blam^{(\eta)}\krog\bxb^{(\eta)}+(\epsilon-\eta)\ones.
\end{equation*}
There exist $\zeta>0$ such that 
$\zeta\bxb^{(\eta)}\leq (\epsilon-\eta)\ones$ and $\xi>0$ such that $\bar r_{\bb}(F^{(\eta)})+\xi<\prod_{i=1}^d(\lambda_i^{(\eta)}+\zeta)^{b_i}.$
We have $(\blam^{(\eta)}+\zeta\ones)\krog\bxb^{(\eta)} \lek F^{(\epsilon)}(\bxb^{(\eta)})$. So, Proposition \ref{Prop536} implies
\begin{equation*}
\bar r_{\bb}(F^{(\eta)})+\xi<\prod_{i=1}^d(\lambda_i^{(\eta)}+\zeta)^{b_i}\leq \bar r_{\bb}(F^{(\epsilon)}).
\end{equation*}
Hence, $\bar r_{\bb}(F^{(\eta)})< \bar r_{\bb}(F^{(\epsilon)})$ for every $0<\eta<\epsilon$.\newline 
Finally, we prove \eqref{abra4}. There exists $C>0$ such that $\by\leq C\ones$ for every $\by\in\Sn_+$. It follows that for every  $0<\epsilon \leq 1$, it holds 
\begin{equation*}
0 \lek\blam^{(\epsilon)}\krog \bx^{(\epsilon)} = F^{(\epsilon)}(\bx^{(\epsilon)})\lek F^{(1)}(\bx^{(\epsilon)}) \lek F^{(1)}(C\ones),
\end{equation*}
and thus $\{\blam^{(\epsilon)}\mid 0<\epsilon\leq 1\}$ is bounded in $\R^d_+$ as well as $\{\bxb^{(\epsilon)}\mid 0<\epsilon\leq 1\}\subset \Sn_+$. Hence, there exists $(\epsilon_k)_{k=1}^{\infty}\subset\R_{++}$ with $\epsilon_k\to 0$, $\bx^{(\epsilon_k)}\to \bx$ and $\blam^{(\epsilon_k)}\to\blam$ as $k\to\infty$. Note that 
$r=\lim_{k\to\infty}\prod_{i=1}^d(\lambda_i^{(\epsilon_k)})^{b_i}=\prod_{i=1}^d\lambda_i^{b_i}.$
Now, \begin{equation*}
F(\bxb^{(\epsilon_k)})=F^{(\epsilon_k)}(\bxb^{(\epsilon_k)})-\epsilon_k\ones=\blam^{(\epsilon_k)}\krog\bxb^{(\epsilon_k)}-\epsilon_k\ones.
\end{equation*}
follows from $F^{(\epsilon_k)}(\bxb^{(\epsilon_k)})=\blam^{(\epsilon_k)}\krog\bxb^{(\epsilon_k)}$. So, by continuity of $F$, we get
\begin{equation*}
F(\bxb)=\lim_{k\to\infty}F(\bxb^{(\epsilon_k)})=\lim_{k\to\infty}\blam^{(\epsilon_k)}\krog\bxb^{(\epsilon_k)}-\epsilon_k\ones=\blam\krog\bxb.
\end{equation*}
On the one hand, by definition, we have 
\begin{equation*}
\bar r_{\bb}(F)\geq\limsup_{m\to\infty}\nnorm{F^{m}(\bxb)}_{\bb}^{1/m}=\limsup_{m\to\infty}\Big(\nnorm{\blam^{\sum_{j=0}^{m-1}A^{j}}\krog\bxb}_{\bb}\Big)^{1/m}=\prod_{i=1}^d\lambda_i^{b_i}.
\end{equation*}
On the other hand, Proposition \ref{Prop536} implies $\blam\krog \bxb = F(\bxb)\lek F^{(\epsilon_k)}(\bxb)$ so that
\begin{equation*}
\bar r_{\bb}(F)\leq \bar r_{\bb}(F^{(\epsilon_k)})=\prod_{i=1}^d(\lambda_i^{(\epsilon_k)})^{b_i}\qquad\forall k\in\N.
\end{equation*}
Letting $k\to\infty$, we finally get $\bar r_{\bb}(F)\leq\prod_{i=1}^d\lambda_i^{b_i}$.
\end{proof}
The proof of Theorem \ref{weak_PF} is now a collection of the results above.
\begin{proof}[Proof of Theorem \ref{weak_PF}]
Theorem \ref{BIGTHM} implies the existence of $(\blam,\bub)\in\R^d_{+}\times \Sn_{+}$ such that $F(\bub)=\blam\krog\bub$ and $\prod_{i=1}^d\lambda_i^{b_i}=\bar r_{\bb}(F)$. Hence, we have $\bar r_{\bb}(F)\leq r_{\bb}(F)$. For the reverse inequality, note that if $\bv\in\Sn_+$ is an eigenvector of $F$ with $F(\bv)=\bt\krog \bv$, then by Proposition \ref{Prop536} we have $\prod_{i=1}^{d}\theta_i^{b_i}\leq \bar r_{\bb}(F)$. It follows that $\bar r_{\bb}(F)= r_{\bb}(F)$ and the characterizations of $r_{\bb}(F)$ follow from Theorem \ref{Thm531}. 
\end{proof}

\section{Existence of positive eigenvectors for nonexpansive mappings}\label{wirr_thm}
As in the linear case, we need to introduce a concept of irreducibility in order to ensure that a nonexpansive mapping has a positive eigenvector. Generalizing the definition of irreducible matrix is a delicate task when dealing with nonlinear mappings. Indeed, already in the study of eigenvectors of order-preserving homogeneous mappings on cones, different generalizations are required to achieve the various results known for irreducible matrices such as existence of a positive eigenvector and simplicity of the spectral radius.  In order to obtain an irreducibility condition which ensures that nonexpansive multi-homogeneous mappings have a positive eigenvector, we choose to extend the graph approach discussed in \cite{Gaubert}. We note however that other existence results are discussed in Sections 6.1, 6.2 and 6.3 of \cite{NB} as well as in Section 5 of \cite{lemmens2018detecting}.

We extend the definition of directed graph associated to order-preserving homogeneous mappings, proposed in \cite{Gaubert}, to multi-homogeneous mappings. 
For any $i\in[d]$ and $j_i\in[n_i]$, consider the mapping $\bu^{(i,j_i)}\colon\R_{+}\to\kone_{+}$ defined as
\begin{equation} \label{defu}
\big(\bub^{(i,j_i)}(t)\big)_{k,l_k}=\begin{cases} t & \text{if } (k,l_k)=(i,j_i)\\ 1 & \text{otherwise},\end{cases} \qquad \forall (k, l_k)\in\I \ =\bigcup_{\nu=1}^d\big(\{\nu\}\times [n_{\nu}]\big).
\end{equation}
Then, the graph associated to $F\in \NB^d$ is given by the following:
\begin{defi}\label{graphdefi}
For $F\in\NB^d$, $\G(F)=(\I,\E)$ is the directed graph with node set $\I$ and such that there is an edge from $(k,l_k)$ to $(i,j_i)$, i.e.\ $\big((k,l_k),(i,j_i)\big)\in\E$, if
\begin{equation*}
\lim_{t\to\infty} F_{k,l_k}\big(\bu^{(i,j_i)}(t)\big) =\infty.
\end{equation*}
\end{defi}

For example, note that if $F(\bx)=M\bx$, for some nonnegative matrix $M\in\R^{n\times n}$, then $\G(F)=(\{1\}\times [n],\E)$ is the graph with $M$ as adjacency matrix. Furthermore, if $G(\bx,\by)=(M\by,M^\top\bx)$ for a nonnegative matrix $M\in\R^{m\times n}$, then $\G(G)=((\{1\}\times[m])\times(\{2\}\times[n]),\E)$ is the bipartite graph with adjacency matrix $\begin{psmallmatrix}0 & M^\top \\ M & 0\end{psmallmatrix}$.

With this definition, we prove the following Theorem \ref{existence_intro}, which generalizes Theorem 2 in \cite{Gaubert}. 
To this end, we proceed as in Section 3.2 of \cite{Gaubert}. 

Let $F\in\NB^d$, $\bb\in\R^d_{++}$, $\G(F)=(\I,\E)$ and, for $r>0$, define
\begin{equation*}
\Psi(r)=\sup\bigg\{ t\geq 0 \ \bigg| \min_{\substack{((i,j_i),(k,l_k))\in\E\\ (a_1,\ldots,a_d)\in\J}} F_{i,j_i}\big(\bu^{(k,l_k)}(t)\big)^{b_i}\prod_{\substack{s=1\\ s\neq i}}^dF_{s,a_s}\big(\bu^{(k,l_k)}(t)\big)^{b_s}\leq r\bigg\}.
\end{equation*}
Note that, by definition of $\G(F)$, $\Psi(r)<\infty$ for any $r>0$ and $\Psi$ is an increasing function. Moreover, note that $\Psi$ has the following property: 

Let $(j_1,\ldots,j_d)\in\J$, $i\in[d]$, $(k,l_k)\in\I$ and $t>0$, if $\big((i,j_i),(k,l_k)\big)\in\E$, then
\begin{equation}\label{psimagic}
\prod_{s=1}^d F_{s,j_s}\big(\bu^{(k,l_k)}(t)\big)^{b_s}\leq r \qquad \text{implies}\qquad t\leq \Psi(r).
\end{equation}

We have 
\begin{thm}\label{existence_intro}
Let $F\in\NB^d$ and $\bb\in\Dn$ be such that $\A(F)^\top\bb=\bb$. If for all $\nu\in[d]$, $l_{\nu}\in[n_\nu]$ and $(j_1,\ldots,j_d)\in\J$, there exists $i_{\nu}\in[d]$ so that there is a path from $(i_\nu,j_{i_{\nu}})$ to $(\nu,l_{\nu})$ in $\G(F)$, then $F$ has an eigenvector in $\Sn_{++}$.
\end{thm}
\begin{proof}
Similarly to the proof of Theorem 6.2.3 \cite{NB}, for the case $d=1$, here we show that for a nonexpansive mapping $F\in\NB^d$ with $d\geq 1$, the assumption on $\G(F)$ in the statement is such that, if the maximal eigenvector $\bx\in\kone_{+,0}$ of $F$ has a zero entry, then the whole vector $\bx_i$ is zero, for some $i\in[d]$, contradicting $\bx\in\kone_{+,0}$.

By Theorem \ref{BIGTHM}, there exists a sequence $(\blam^{(\epsilon_k)},\bxb^{(\epsilon_k)})_{k=1}^{\infty}\subset \R^d_{++}\times \Sn_{++}$ such that 
$\lim_{k\to\infty}(\blam^{(\epsilon_k)},\bxb^{(\epsilon_k)}) =(\blam,\bxb^*) $ and $ F(\bxb^*)=\blam \krog \bx^*\in\Sn_{+}.$
Since $\blam^{(\epsilon_k)}\to\blam$, there exists a constant $M_0>0$ such that 
\begin{equation}\label{defM0}
\prod_{s=1}^d(\lambda_s^{(\epsilon_k)})^{b_s}\leq M_0 \qquad \forall k \in \N.
\end{equation}
Suppose by contradiction that $\bxb^*\in\Sn_{+}\setminus\Sn_{++}$. By taking a subsequence if necessary, we may assume that there exists $(j_1,\ldots,j_d)\in\J$ and $\omega\in[d]$ such that $\min_{t_s\in[n_s]}x_{s,t_s}^{(\epsilon_k)}=x^{(\epsilon_k)}_{s,j_s}$, $\forall s\in[d], k \in\N$, and $\lim_{k\to\infty} x^{(\epsilon_k)}_{\omega,j_\omega}=x_{\omega,j_\omega}^*=0$.
By the compactness of $\Sn_+$, there exists $\t C>0$ such that $y_{s,t_s}\leq \t C$ for all $\by\in \Sn_+$. It follows that
\begin{equation}\label{contra1}
0\leq \lim_{k\to \infty} \prod_{s=1}^d(x^{(\epsilon_k)}_{s,j_s})^{b_s} \leq \t C^{1-b_{\omega}}\lim_{k\to\infty} (x^{(\epsilon_k)}_{\omega,j_\omega})^{b_\omega}=0.
\end{equation}
Since $\bx^*\in\Sn_{+}$, there exists $(l_1,\ldots,l_d)\in\J$ with $x_{s,l_s}^*>0$ for all $s\in[d]$. Thus,
\begin{equation}\label{contra2}
\lim_{k\to \infty} \prod_{s=1}^d(x^{(\epsilon_k)}_{s,l_s})^{b_s} =\prod_{s=1}^d(x^*_{s,l_s})^{b_s}>0.
\end{equation}
Let $\nu\in[d]$, by assumption on $\G(F)$, there exists $i_{\nu}\in[d]$ and a path $(i_\nu,j_{i_\nu})=(m_1,\xi_{m_1})\to (m_2,\xi_{m_2})\to\ldots\to (m_{N_{\nu}},\xi_{m_{N_{\nu}}})=(\nu,l_\nu)$ in $\G(F)$ with $N_{\nu}\leq n_1+\ldots+n_d$. Define $\bi{(1)}, \bi{(2)},\ldots,\bi{(N_{\nu})}\in\J$ as
\begin{equation*}
\bi_s{(a)}=\begin{cases} \xi_{m_a}&\text{if } s= m_a, \\ j_s &\text{otherwise.}\end{cases} \qquad \forall s \in [d], \ a\in[N_{\nu}].
\end{equation*}
Fix $k\in\N$ and let $t=x_{m_2,\xi_{m_2}}^{(\epsilon_k)}/x_{m_2,j_{m_2}}^{(\epsilon_k)}$ and $\bal= \big((x^{(\epsilon_k)}_{1,j_1})^{-1},\ldots,(x^{(\epsilon_k)}_{d,j_d})^{-1}\big)$. \red{Note that 
$\bal$ is well defined since Theorem \ref{BIGTHM} ensures that $\bx^{(\epsilon_k)}\in\Sn_{++}$ for all $k$.} We have $\bu^{(m_2,\xi_{m_2})}(t)\lek \bal\krog\bx^{(\epsilon_k)}$ and thus $F(\bu^{(m_2,\xi_{m_2})}(t))\lek F(\bal\krog\bx^{(\epsilon_k)})$. Furthermore, as $F(\by)\lek F^{(\epsilon_k)}(\by)$ for all $\by\in \kone_+$, it holds $F(\bal\krog\bx^{(\epsilon_k)})\lek F^{(\epsilon_k)}(\bal\krog\bx^{(\epsilon_k)})$. It follows that
\begin{align}\label{derrrrr}
(x_{m_1,j_{m_1}}^{(\epsilon_k)})^{b_{m_1}}&\prod_{s=1}^dF_{s,\bi_s(1)}\big(\bu^{(m_2,\xi_{m_2})}(t)\big)^{b_s}\leq (x_{m_1,j_{m_1}}^{(\epsilon_k)})^{b_{m_1}}\prod_{s=1}^dF_{s,\bi_s{(1)}}(\bal\krog\bx^{(\epsilon_k)})^{b_s}\notag\\
&\!\!\!\!\!\!\!\!\!\!\!\!\!\!\!\!\!\!\!\!\!\!\!\!\!\!\!\!\!\!\!\!\!\leq (x_{m_1,j_{m_1}}^{(\epsilon_k)})^{b_{m_1}}\prod_{s=1}^dF^{(\epsilon_k)}_{s,\bi_s{(1)}}(\bal\krog\bx^{(\epsilon_k)})^{b_s}=\bigg(\prod_{\substack{s=1,\\ s \neq m_1}}^d (x^{(\epsilon_k)}_{s,j_s})^{b_s}\bigg)^{-1}\prod_{s=1}^dF^{(\epsilon_k)}_{s,\bi_s{(1)}}(x^{(\epsilon_k)})^{b_s}\notag\\
&\!\!\!\!\!\!\!\!\!\!\!\!\!\!\!\!\!\!\!\!\!\!\!\!\!\!\!\!\!\!\!\!\!=\bigg(\prod_{\substack{s=1,\\ s \neq m_1}}^d (x^{(\epsilon_k)}_{s,j_s})^{b_s}\bigg)^{-1}\prod_{s=1}^d(\lambda_s^{(\epsilon_k)}\bx_{s,\bi_s{(1)}}^{(\epsilon_k)})^{b_s}
=(x_{m_1,j_{m_1}}^{(\epsilon_k)})^{b_{m_1}}\prod_{s=1}^d(\lambda_s^{(\epsilon_k)})^{b_s}\notag\\ 
&\phantom{\!\!\!\!\!\!\!\!\!\!\!\!\!\!\!\!\!\!\!\!\!\!\!\!\!\!\!\!\!\!\!\!\!=\prod_{\substack{s=1,\\ s \neq m_1}} (x^{(\epsilon_k)}_{s,j_s})^{b_s}\prod_{s=1}(\lambda_s^{(\epsilon_k)}\bx_{s,\bi_s{(1)}}^{(\epsilon_k)})^{b_s}}\qquad\quad\leq (x_{m_1,j_{m_1}}^{(\epsilon_k)})^{b_{m_1}}M_0,
\end{align}
where $M_0>0$ satisfies \eqref{defM0}. Hence, by \eqref{psimagic}, $t=x_{m_2,\xi_{m_2}}^{(\epsilon_k)}/x_{m_2,j_{m_2}}^{(\epsilon_k)} \leq \Psi(M_0)$ and
\begin{equation*}
\prod_{s=1}^d (x^{(\epsilon_k)}_{s,\bi_s(2)})^{b_s} \leq  M_1\prod_{s=1}^d (x_{s,j_s}^{(\epsilon_k)})^{b_s}\qquad \text{with} \qquad M_1 =\Psi(M_0)^{b_{m_2}}.
\end{equation*}
Applying this procedure again to $ (m_3,\xi_{m_3})$, we get the existence of a constant $M_2>0$ independent of $k$, such that 
\begin{equation*}
\prod_{s=1}^d (x^{(\epsilon_k)}_{s,\bi_s(3)})^{b_s} \leq  M_2\prod_{s=1}^d (x^{(\epsilon_k)}_{s,j_s})^{b_s}.
\end{equation*}
Indeed, let $t=x^{(\epsilon_k)}_{m_3,\xi_{m_3}}/x_{m_3,j_{m_3}}^{(\epsilon_k)}$, then $\bu^{(m_3,\xi_{m_3})}(t)\lek \bal\krog\bx^{(\epsilon_k)}$ and, similarly to \eqref{derrrrr}, we get
\begin{align*}
&(x_{m_2,j_{m_2}}^{(\epsilon_k)})^{b_{m_2}}\prod_{s=1}^dF_{s,\bi_s(2)}\big(\bu^{{(m_3,\xi_{m_3})}}(t)\big)^{b_s}\leq (x_{m_2,j_{m_2}}^{(\epsilon_k)})^{b_{m_2}}\prod_{s=1}^dF_{s,\bi_s{(2)}}(\bal\krog\bx^{(\epsilon_k)})^{b_s}\\
&  \leq (x_{m_2,j_{m_2}}^{(\epsilon_k)})^{b_{m_2}}\prod_{s=1}^dF^{(\epsilon_k)}_{s,\bi_s{(2)}}(\bal\krog\bx^{(\epsilon_k)})^{b_s}=\bigg(\prod_{\substack{s=1,\\ s \neq m_2}}^d (x^{(\epsilon_k)}_{s,j_s})^{b_s}\bigg)^{-1}\prod_{s=1}^dF^{(\epsilon_k)}_{s,\bi_s{(2)}}(\bx^{(\epsilon_k)})^{b_s}\\
&  =\!\! \bigg(\!\prod_{\substack{s=1,\\ s \neq m_2}}^d (x^{(\epsilon_k)}_{s,j_s})^{b_s}\!\bigg)^{\!\!\!-1}\!\!\prod_{s=1}^d\big(\lambda_sx^{(\epsilon_k)}_{s,\bi_s{(2)}}\big)^{b_s}\!=\!(x^{\epsilon_k}_{m_2,\xi_{m_2}})^{b_{m_2}}\!\prod_{s=1}^d(\lambda_s^{(\epsilon_k)})^{b_s} \! \leq \!  (x^{(\epsilon_k)}_{m_2,j_{m_2}})^{b_{m_2}} M_1M_0.
\end{align*}
Hence, with $M_2=\Psi(M_0M_1)^{b_{m_3}}$, we get $(x_{m_3,\xi_{m_3}}^{(\epsilon_k)}/x_{m_3,j_{m_3}}^{(\epsilon_k)})^{b_{m_3}} \leq M_2$ which implies the desired inequality. Repeating this process at most $N_{\nu}$ times, we obtain $C_{\nu}>0$ independent of $k$, such that 
\begin{equation}\label{onthewayeqnnnn}
(x^{(\epsilon_k)}_{\nu,l_{\nu}})^{b_\nu}\prod_{\substack{s=1\\ s\neq \nu}}^d (x^{(\epsilon_k)}_{s,j_s})^{b_s}=\prod_{s=1}^d (x^{(\epsilon_k)}_{s,\bi_s(N_{\nu})})^{b_s} \leq  C_{\nu}\prod_{s=1}^d (x^{(\epsilon_k)}_{s,j_s})^{b_s} \qquad \forall k\in\N.
\end{equation}
Taking the product over $\nu\in[d]$ in \eqref{onthewayeqnnnn} and dividing by $\prod_{s=1}^d (x^{(\epsilon_k)}_{s,j_s})^{(d-1)b_s}$ shows
\begin{equation*}
\prod_{\nu=1}^d(x^{(\epsilon_k)}_{\nu,l_{\nu}})^{b_\nu}\leq C\prod_{s=1}^d (x^{(\epsilon_k)}_{s,j_s})^{b_s} \qquad\forall k\in\N,
\end{equation*}
where $C=\prod_{\nu=1}^dC_{\nu}.$ Finally, using \eqref{contra1} and \eqref{contra2} we get a contradiction.
\end{proof}

As noted in Corollary 6.2.4 \cite{NB} for the case $d=1$, there exists a dual version of Theorem \ref{existence_intro} which follows by considering the mapping $\tau \colon \R^{N}_{++}\to \R_{++}^N$ defined as $\tau(\bz)=(z_1^{-1},\ldots,z_N^{-1})$ with $N=n_1+\ldots+n_d$. 
More precisely, let $F\in\NB^d$ and define $\hat F\colon\kone_{++}\to\kone_{++}$ as 
$\hat F(\bx)=\tau\big(F(\tau(\bx))\big)$ for all $\bx\in\kone_{++}.$
Then, $\tau$ is a bijection between the positive eigenvectors of $F$ and $\hat F$. Moreover, by Theorem \ref{extend}, $\hat F$ can be continuously extended on $\kone_{+}$ so that $\hat{F}\in\NB^d$. Note that if $\bu^{(k,j_k)}(t)$ is defined as in \eqref{defu}, and $\G(\hat F)=(\I,\hat E)$, then we have $\big((k,l_k),(i,j_i)\big)\in\hat\E$, if and only if
$F_{k,l_k}\big(\bu^{(i,j_i)}(t)\big)\to 0$ as $t\to 0$. The following corollary is a direct consequence of Theorem \ref{existence_intro} applied to $\hat F$. 
\begin{cor}\label{existdual}
Let $F\in\NB^d$, $A=\A(F)$ and $\bb\in\Dn$ with $A^\top\bb=\bb$. Let $\hat F$ be defined as above. Suppose that, for every $(\nu,l_\nu)\in\I$ and $(j_1,\ldots,j_d)\in\J$ there exists $i_\nu\in[d]$ such that there is a path from $(i_\nu,j_{i_\nu})$ to $(\nu,l_\nu)$ in $\G(\hat F)$. Then $F$ has an eigenvector in $\Sn_{++}$.
\end{cor}

We conclude with some important observations. First note that, when $d=1$, the graph of Definition \ref{graphdefi} coincides with the one proposed in \cite{Gaubert} and our existence Theorem \ref{existence_intro} coincides with Theorem 2 \cite{Gaubert} where the graph is required to be strongly connected. However, when $d>1$, there are mappings having a graph which is not strongly connected but satisfy the assumptions of Theorem \ref{existence_intro}. Such a mapping is discussed in the following example.
\begin{ex}\label{nonirrgraph}
Let $n_1=n_2=2$ and $F\in \NB^2$ with
\begin{equation*}
F((s,t),(u,v))=\Big(\begin{pmatrix}\min\{su,sv\}^{1/4}\\\min\{tu,tv\}^{1/4}\end{pmatrix},\begin{pmatrix}\max\{su,tv\}^{1/4}\\ \max\{sv,tu\}^{1/4}\end{pmatrix}\Big) 
\end{equation*}
Then, $F$ has homogeneity matrix $A=\tiny{\tfrac{1}{4}\begin{pmatrix} 1& 1 \\ 1 & 1\end{pmatrix}}$ and $F(\ones,\ones)=(\ones,\ones)$ and the graphs $\G(F),\G(\hat F)$ are given by
\begin{equation*}
\begin{array}{c}
\begin{tikzpicture}[baseline= (a).base]
\node[scale=.8] (a) at (0,0){
\begin{tikzcd}[cells={nodes={circle,draw,font=\sffamily\Large\bfseries}},thick]
 \arrow[loop left, distance=3em, start anchor={[yshift=-1.5ex]west}, end anchor={[yshift=1.5ex]west},thick]{}{} s  
 &\arrow[loop right, distance=3em, start anchor={[yshift=1.5ex]east}, end anchor={[yshift=-1.5ex]east}, thick]{}{}t\\
 u \ar[thick]{r} \ar[thick]{ru}\arrow[loop left, distance=3em, start anchor={[yshift=-1.5ex]west}, end anchor={[yshift=1.5ex]west},thick]{}{} \ar[thick]{u}
  & v \arrow[thick]{u}\ar[thick]{l} \ar[thick]{lu}\arrow[loop right, distance=3em, start anchor={[yshift=1.5ex]east}, end anchor={[yshift=-1.5ex]east}, thick]{}{}
\end{tikzcd}
};
\end{tikzpicture}\\
\G(F)\end{array}
\qquad \begin{array}{c}\begin{tikzpicture}[baseline= (a).base]
\node[scale=.8] (a) at (0,0){
	\begin{tikzcd}[cells={nodes={circle,draw,font=\sffamily\Large\bfseries}},thick]
	\arrow[loop left, distance=3em, start anchor={[yshift=-1.5ex]west}, end anchor={[yshift=1.5ex]west},thick]{}{} s  \ar[thick]{rd}\ar[thick]{d} 
	&\arrow[loop right, distance=3em, start anchor={[yshift=1.5ex]east}, end anchor={[yshift=-1.5ex]east}, thick]{}{}t \arrow[thick]{d}\ar[thick]{ld}\\
	u  
	& v
	\end{tikzcd}
};
\end{tikzpicture}
\\
\G(\hat F)\end{array}
\end{equation*}
In particular, note that neither $\G(F)$ nor $\G(\hat F)$ is strongly connected but $\G(F)$ satisfies the assumptions of Theorem \ref{existence_intro}. Furthermore, by replacing the $\min$'s by $\max$'s and the $\max$'s by $\min$'s in the definition of $F$, we obtain a mapping $H$ such that $\G(F)=\G(\hat H)$ and $\G(\hat F)=\G(H)$. $H$ is then an example of mapping such that neither $\G(H)$ nor $\G(\hat H)$ is strongly connected but $\G(\hat H)$ satisfies the assumptions of Corollary~\ref{existdual}.  
\end{ex}
Finally, let us note that unlike the linear case, the assumption that $\G(F)$ is strongly connected does not imply the uniqueness of positive eigenvectors. This is shown by the following example 
\begin{ex}\label{nonuniquevect}
Let $\epsilon \in (0,1)$, $d=1$, $n_1=3$ and $F\in \NB^d$ with
$$F(a,b,c)=\big(\max(a,b,c), \max(\epsilon a,b), \max(\epsilon b,c)\big).$$
Then, $\G(F)$ is strongly connected and $(1,b,c)$ is an eigenvector of $F$ for all $b,c\in [\epsilon,1]$.
\end{ex}

\section{Maximality and uniqueness of positive eigenvectors}\label{CW_M_U}
Theorems \ref{Banachcombi} and \ref{existence_intro} provide sufficient conditions for the existence of a positive eigenvector. In the linear case, it is known that the eigenvalue associated to a positive eigenvector of a nonnegative matrix always coincides with its spectral radius. This can be deduced by the Collatz-Wielandt formula. A generalization of this characterization to the spectral radius of nonexpansive mappings in $\NB^1$ can be found in Theorem 5.6.1 \cite{NB} and Theorem 1 \cite{GV12}. In the context of nonnegative multi-linear forms, Collatz-Wielandt formulas were established for different types of spectral problems \cite{Fried, us, Yang2}. By combining techniques from the proofs of Theorem 5.6.1 in \cite{NB} and of Theorem 1 in \cite{us}, we obtain the following Collatz-Wielandt characterization of the spectral radius for mappings in $\NB^d$:
\begin{thm}\label{CW_intro}
Let $F\in\NB^d$, $A=\A(F)$ and $\bb\in\Dn$. If either $A^\top\bb = \bb$ or $\rho(A)<1$ and $(A^\top\bb)_i\leq b_i$ for every $i\in[d]$, then
\begin{equation} \label{CWeq1intro}
\inf_{\bu\in\Sn_{++}}\prod_{i=1}^d \bmaxi{i}{F(\bu)}{\bu}^{b_i} = r_{\bb}(F) = \max_{\bv\in\Sn_{+}}\prod_{i=1}^d \bmini{i}{F(\bv)}{\bv}^{b_i}.
\end{equation}
\end{thm}
The proof of this result is postponed to the end of Subsection \ref{CW_M_U1}.

In particular, we note that if $d=1$ and $F$ is linear, then the left hand side of \eqref{CWeq1intro} reduces to the classical Collatz-Wielandt formula for nonnegative matrices . 

Our next result is concerned with the simplicity of the positive eigenvector of a multi-homogeneous mappings and its eigenvalue. In the linear case, it is known that every nonnegative irreducible matrix has a unique real eigenvector corresponding to its spectral radius and this vector must have positive entries. We have seen in Theorem \ref{existence_intro} a possible way to generalize the notion of irreducibility to mappings in $\NB^d$ which ensures existence of a positive eigenvector. However, as shown in Example \ref{nonuniquevect}, already in the case $d=1$, this assumption does not guarantee  that this positive eigenvector is unique in $\Sn_{++}$. This suggests that the notion of irreducibility needs to be generalized in a different way in order to obtain uniqueness results. A possible approach is proposed in Theorem 2.5 \cite{Nussb} and Theorem 6.1.7 \cite{NB}, which have assumptions on the derivative of the mapping. More precisely, let $F\in\NB^1$ be such that $\A(F)=1$, $F$ has a positive eigenvector $\bu\in\Sn_{++}$ and $F$ is differentiable at $\bu\in\S_{++}$. Recall that $DF(\bu)$ denotes the Jacobian of $F$ at $\bu$. If $DF(\bu)$ is irreducible, then Theorem 2.5 \cite{Nussb} implies that $\bu$ is the unique eigenvector of $F$ in $\Sn_{++}$ and Theorem 6.1.7 \cite{NB} implies that for any eigenvector $\bv\in\Sn_{+}\setminus\Sn_{++}$ with $F(\bv)=\theta\bv$ we have $\theta<r_1(F)$. The combination of these results can therefore be interpreted as a result on the simplicity of the spectral radius. Indeed the first one implies that the positive eigenvector is unique and the second one implies that the spectral radius of $F$ can only be attained by a positive eigenvector. The following theorem generalizes the results above to the multi-homogeneous setting.
\begin{thm}\label{simplicity}
Let $F\in\NB^d$, $A=\A(F)$ and $\bb\in\Dn$. Suppose that $F$ has a positive eigenvector $\bu\in\Sn_{++}$. Then, $\bu$ is the unique eigenvector of $F$ in $\Sn_{++}$ if either $\rho(A)<1$ or $A^\top\bb=\bb$, $F$ is differentiable at $\bu$ and $DF(\bu)$ is irreducible. Furthermore, suppose that $F$ has an eigenvector $\bv\in\S_{+}\setminus\S_{++}$ and let $\bt\in\R^d_{+}$ be such that $F(\bv)=\bt\krog\bv$. If $A^\top\bb=\rho(A)\bb$, $F$ is differentiable at $\bu$ and $DF(\bu)$ is irreducible, then $\prod_{i=1}^d\theta_i^{b_i}<r_{\bb}(F)$.
\end{thm}
The proof of this theorem is postponed to the end of Subsection \ref{subsec:uniqueness}.

It turns out that the assumptions in the theorem above can be refined. On the one hand, as in Theorem 2.5 \cite{Nussb}, in order to guarantee the uniqueness of a positive eigenvector the requirement that $DF(\bu)$ is irreducible can be relaxed to a condition on the eigenspace of $DF(\bu)$ corresponding to its spectral radius. On the other hand, for $d>1$, it can be shown that the spectral radius can not be attained in $\Sn_{+}\setminus\Sn_{++}$ under a weaker assumption than irreducibility. These relaxed assumptions are given in Theorems \ref{unique} and \ref{rad<} in Subsection \ref{subsec:uniqueness}.

\subsection{Collatz-Wielandt formulas}\label{CW_M_U1}
For convenience in the proof of Theorem \ref{CW_intro}, for a given $\bb\in\Dn$, we introduce the functions $\cwl\colon\NB^d\times \kone_{+,0}\to\R_{+}$ and $\cwu\colon\NB^d\times \kone_{++}\to\R_{++}$ defined as
\begin{equation}\label{def_cwu}
\cwu(F,\bu)=\prod_{i=1}^d \Big(\max_{j_i\in[n_i]}\frac{F_{i,j_i}(\bu)}{u_{i,j_i}}\Big)^{b_i},\qquad \cwl(F,\bx)=\prod_{i=1}^d \Big(\min_{\substack{j_i\in[n_i]\\ x_{i,j_i}>0}}\frac{F_{i,j_i}(\bx)}{x_{i,j_i}}\Big)^{b_i}.
\end{equation}
With this notation, the characterization of Theorem \ref{CW_intro} can be reformulated as
\begin{equation}\label{local_CW}
\inf\big\{\cwu(F,\bu)\ \big|\ {\bu\in\Sn_{++}} \big\} = r_{\bb}(F) = \max\big\{\cwl(F,\bv)\ \big|\ {\bv\in\Sn_+}\big\}.
\end{equation}
Note also that for $F\in\NB^d$, $\bb\in\R^d_{++}$ and $\bx\in\kone_{++}$, it holds $\cwl(F,\bx)=\cwu(F,\bx)$ if and only if $\bx$ is an eigenvector of $F$.

The proof of Theorem \ref{CW_intro} contains two cases, namely the case where $F\in\NB^d$ is nonexpansive and the one where $F$ is a strict contraction. For the first case we generalize Theorem 5.6.1 in \cite{NB} which holds for the case $d=1$. For the second case, we generalize the Collatz-Wielandt formula of Theorem 21 in \cite{us}.
\begin{proof}[Proof of Theorem \ref{CW_intro}]
First assume that $A^\top \bb=\bb$. Let $\bx\in\Sn_{++}$ and $k\in\N$, then we have $F^{k}(\bx)\ \lek\ \Bmax{F(\bx)}{\bx}^{\sum_{j=0}^{k-1} A^j}\krog\bx.$ 
Proposition \ref{Prop536} implies
\begin{equation*}
r_{\bb}(F)=\lim_{k\to\infty}\nnorm{F^{k}(\bx)}_{\bb}^{1/k} \leq \lim_{k\to\infty}\prod_{i=1}^d\bmaxi{i}{F(\bx)}{\bx}^{(\frac{1}{k}\sum_{j=0}^{k-1} A^j\bb)_i}=\cwu(F,\bx).
\end{equation*}
Hence, $r_{\bb}(F)\leq \inf\{\cwu(F,\bu)\mid {\bu\in\Sn_{++}}\}$. To show equality, assume first that $F$ has an eigenvector $\bu\in\Sn_{++}$. Then $\cwu(F,\bu)=r_{\bb}(F)$ and we are done. Now, suppose that $F$ does not have an eigenvector in $\Sn_{++}$, let $F^{(\delta_k)}$ and $(\blam^{(\delta_k)},\bx^{(\delta_k)})\in\R^d_{++}\times \Sn_{++}$ be as in Theorem \ref{BIGTHM}. Note that $\cwu(F,\bx)\leq \cwu(F^{(\delta_k)},\bx)$ as $F(\bx)\lek F^{(\delta_k)}(\bx)$ for every $k\in\N$ and $\bx\in\kone_{++}$.
It follows that
\begin{equation*}
r_{\bb}(F)=\lim_{k\to\infty}r_{\bb}\big(F^{(\delta_k)}\big) =\lim_{k\to\infty} \inf_{\bx\in\Sn_{++}}\cwu(F^{(\delta_k)},\bx) \geq \inf_{\bx\in\Sn_{++}}\cwu(F,\bx).
\end{equation*}
Now, we prove $r_{\bb}(F) = \max\{\cwl(F,\bv)\mid {\bv\in\Sn_+}\}$. To this end, let $\by\in\Sn_+$, if there exists $(i,j_i)\in\I$ such that $y_{i,j_i}>0$ and $F_{i,j_i}(\by)=0$, then $\cwl(F,\by)=0 \leq r_{\bb}(F)$. If this is not the case, then we have $\bt\krog\by\leq F(\by)$ with $\bt\in\R^d_{++}$ defined as $\theta_i = \min\{F_{i,j_i}(\by)/y_{i,j_i}\mid y_{i,j_i}>0, \ j_i\in[n_i] \}$ for all $i \in[d]$.
Hence, by Proposition \ref{Prop536}, we get
$\cwl(F,\by) = \prod_{i=1}^d \theta_i^{b_i}\leq r_{\bb}(F).$
Finally, by Theorem \ref{BIGTHM}, we know that there exists $(\blam,\bu)\in\R^d_{+}\times \Sn_{+}$ such that $r_{\bb}(F)=\prod_{i=1}^d\lambda_i^{b_i}=\cwl(F,\bu)$.

Now, suppose that $\rho(A)<1$ and $A^\top\bb\leq \bb$. As $\rho(A)<1$, Theorem \ref{Banachcombi} implies the existence of $(\blam,\bub)\in\R^d_{++}\times\Sn_{++}$ such that $F(\bub)=\blam\krog\bub$. Clearly, we have $\cwl(F,\bub)=\cwu(F,\bub)=r_{\bb}(F)$.
To prove the right-hand side of \eqref{local_CW}, it suffices to prove that for every $\byb\in\Sn_{+}$, we have $\cwl(F,\byb)\leq r_{\bb}(F)$. So, let $\byb\in\Sn_{+}$, if there exists $(i,j_i)\in\I$ such that $y_{i,j_i}>0$ and $F_{i,j_i}(\by)=0$, then the inequality is clear. Thus, we may assume without loss of generality that 
$F_{i,j_i}(\by)>0$ for every $(i,j_i)\in\I$ such that $y_{i,j_i}>0$.
Let $\bt\in\R^d_{++}$ be defined as 
$\theta_i = \min\{u_{i,j_i}/y_{i,j_i}\mid y_{i,j_i}>0, \ j_i\in[n_i] \}$ for all $i \in[d].$
Then $\bt\leq \ones$ because $\theta_i= \norm{\theta_i\by_i}_{\gamma_i}\leq \norm{\bu_i}_{\gamma_i}=1$ for all $i \in[d].$
Let $\boldsymbol{\Theta} = (\theta_1^{-1},\dots,\theta_d^{-1})$, then $\boldsymbol{\Theta}\geq \ones$ and $\byb\lek\boldsymbol{\Theta}\krog\bub.$
Thus, for $\bs= \bb-A^\top\bb\in\R^d_{+}$, we have
$\prod_{i=1}^d\Theta_i^{-\s_i}\leq 1.$
Now, note that
$F\big(\boldsymbol{\Theta}\krog\bub\big)=\big(\blam\circ\boldsymbol{\Theta}^{A}\big)\krog \bub$ and thus
\begin{equation*}
\cwl(F,\byb)\leq\prod_{i=1}^d\Big(\min_{\substack{j_i\in[n_i]\\ y_{i,j_i}>0}}\frac{F_{i,j_i}(\boldsymbol{\Theta}\krog\bu)}{y_{i,j_i}}\Big)^{b_i}
=\prod_{i=1}^d\Theta_i^{-\s_i}\lambda_i^{b_i}\leq r_{\bb}(F).
\end{equation*}
The left-hand side of \eqref{local_CW} can be proved in a similar way. Indeed, if $\byb\in\Sn_{++}$, then 
\begin{equation*}
\cwu(F,\bu)\geq\prod_{i=1}^d\bmaxi{i}{F(\Bmin{\by}{\bub}\krog\bub)}{\by}^{b_i} 
=\prod_{i=1}^d\bmini{i}{\by}{\bub}^{-\s_i}\lambda_i^{b_i}\geq r_{\bb}(F), 
\end{equation*}
as $\prod_{i=1}^d\bmini{i}{\byb}{\bub}^{-\s_i}\geq 1$.
\end{proof}
\subsection{Uniqueness and simplicity of positive eigenvectors}\label{subsec:uniqueness}
We prove the following theorem which gives a condition ensuring that the eigenvalue corresponding to an eigenvector which has some zero entry can not be maximal.
\begin{thm} \label{rad<}
Let $F\in\NB^d$ and $A=\A(F)$. Suppose that there exists $\bb\in\Dn$, $\blam\in\R^d_{++}$ and $\bu\in\Sn_{++}$ such that $A^\top\bb=\rho(A)\bb$ and $F(\bu)=\blam\krog\bu$. Assume $\rho(A)\leq 1$, $F$ is differentiable at $\bub$ and there exist $i\in[d]$ and $\tau\in\N$ such that 
\begin{equation}\label{dirr}
\forall \bw\in\kone_{+}\saufzero, \quad \text{if} \quad \bx = \sum_{k=1}^\tau DF(\bu)^k\bw, \quad \text{then} \quad \bx_i\in\R^{n_i}_{++}.
\end{equation}
Then, for every eigenpair $(\bt,\bvb)\in\R^d_{+}\times(\Sn_{+}\sauf\kone_{++})$ with $F(\bvb)=\bt\krog\bvb$, it holds
$\prod_{j=1}^d\theta_j^{b_j}<\prod_{j=1}\lambda_j^{b_j}.$
\end{thm}
Before giving a proof of this theorem, we note that while in the case $d=1$ the irreducibility assumption \eqref{dirr} is equivalent to requiring $DF(\bu)$ to be irreducible, this is not the case anymore when $d>1$. Indeed, if $DF(\bu)$ is irreducible, then \eqref{dirr} is satisfied, however the converse might not be true as shown by the following example. In fact, for any $d\geq 1$, $DF(\bu)$ is irreducible if and only if \eqref{dirr} holds and $\A(F)$ is irreducible.
\begin{ex}\label{irrrrrr_ex}
Let $d=2$, $n_1=n_2=2$ and $F\in\NB^d$ with
\begin{equation*}
F((s,t),(u,v))=\Big(\big((st)^{1/4}u^{1/2},(st)^{1/4}v^{1/2}\big),\big((uv)^{1/2}, (uv)^{1/2}\big)\Big).
\end{equation*}
Then, $F(\ones)=\ones$, Theorem \ref{rad<} applies to $F$, but $DF(\ones)$ is not irreducible.
\end{ex}
We now prove Theorem \ref{rad<}. The techniques used are inspired by the proof of Theorem 6.1.7 in \cite{NB} which implies the same result for the case $d=1$.
\begin{proof}[Proof of Theorem \ref{rad<}]
Let $\norm{\cdot}$ be any norm on $\R^{n_1}\times \ldots\times \R^{n_d}$ and $\blam\in\R^d_{++}$ be such that $F(\bu)=\blam\krog\bu$. We first prove the statement for $\blam=\ones$, then we show how to transfer the proof to the case $\blam\neq\ones$. By the chain rule, we have $DF(\bub)^k=DF^{k}(\bub)$ for every $k\in\N$. Suppose by contradiction that there exists $(\bt,\bvb)\in\R^d_{+}\times(\Sn_{+}\sauf\kone_{++})$ with $F(\bvb)=\bt\krog\bvb$ and $\prod_{l=1}^d\theta_l^{b_l}=1$. Let $\bal \in\R^d_{++}$ be defined as
$\alpha_k=\min\big\{{u_{k,l_k}}/{v_{k,l_k}}\ \big|\ l_k\in[n_k], \ v_{k,l_k}>0\big\}$ for every $k\in[d]$,
then  $0\lekk\bub-\bal\krog\bvb\lek\bu$. Hence
$-\big(\sum_{k=1}^\tau DF(\bub)^k(\bal\krog\bvb-\bub)\big)_i \in\R^{n_i}_{++}.$
For $t\in (0,1]$, define $\byb(t)=(1-t)\bub+t\bal\krog\bvb\lekk\bub$ and note that
\begin{equation*}
F^{k}\big(\byb(t)\big)=F^{k}(\bub)+t\ DF(\bub)^k(\bal\krog\bvb-\bub) + t\ \norm{\bal\krog\bvb-\bub}\ \epsilon_k\big(t(\bal\krog\bvb-\bub)\big)
\end{equation*}
where $\lim_{\norm{\bwb}\to 0}\epsilon_k(\bwb)=0$. If follows that, with $\bzb=\bal\krog\bvb-\bub$, we have
\begin{align*}
\sum_{k=1}^\tau \Big(F^k(\bub)-F^k\big(\byb(t)\big)\Big)= t \bigg(-\sum_{k=1}^\tau DF(\bu)^k\bzb -\norm{\bzb}\sum_{k=1}^\tau\epsilon_k\big(t\bzb\big)\bigg).
\end{align*}
Since $\lim_{t\to 0}\sum_{k=1}^\tau\epsilon_k\big(t\bzb\big)=0$ and $-\sum_{k=1}^\tau \big(DF(\bu)^k\bzb\big)_i\in\R^{n_i}_{++}$, there exists $s\in (0,1]$ such that for every $t\in (0,s]$, it holds
$\sum_{k=1}^\tau \Big(F_i^k(\bub)-F_i^k\big(\byb(t)\big)\Big)\in\R^{n_i}_{++}.$
For all $t\in (0,1]$, we have $\bal\krog\bvb \lek\byb(t)$ and thus 
$\sum_{k=1}^\tau \big(F^k\big(\byb(t)\big)-F^k(\bal\krog\bv)\big)\in\kone_{+}.$
It follows with $\blam = \ones$ and $F(\bu)=\bu$ that
$\sum_{k=1}^{\tau}F^k(\bal\krog\bv) \lek \tau \bu$ and $ \sum_{k=1}^{\tau}F_i^k(\bal\krog\bv)< \tau\bu_i.$
So, for every $(j_1,\ldots,j_d)\in\J$, we have
\begin{equation*}
\tau\prod_{l=1}^du_{l,j_l}^{b_l}> \prod_{l=1}^d\bigg(\sum_{k=1}^{\tau}F^k_{l,j_l}(\bal\krog\bv)\bigg)^{b_l}=\prod_{l=1}^dv_{l,j_l}^{b_l}\bigg(\sum_{k=1}^\tau \big(\bal^{A^k}\big)_l\big(\bt^{\sum_{s=0}^{k-1}A^s}\big)_l\bigg)^{b_l}.
\end{equation*}
Using the inequality relating arithmetic and geometric mean, we get
\begin{equation*}
\sum_{k=1}^\tau \big(\bal^{A^k}\big)_l\big(\bt^{\sum_{s=0}^{k-1}A^s}\big)_l\geq \tau \prod_{k=1}^\tau \Big(\big(\bal^{A^k}\big)_l\big(\bt^{\sum_{s=0}^{k-1}A^s}\big)_l\Big)^{1/\tau}.
\end{equation*}
It follows that
\begin{align*}
\prod_{l=1}^d\bigg(\sum_{k=1}^\tau \big(\bal^{A^k}\big)_l&\big(\bt^{\sum_{s=0}^{k-1}A^s}\big)_l\bigg)^{b_l} \geq \tau \prod_{k=1}^\tau \prod_{l=1}^d\Big(\big(\bal^{A^k}\big)_l\big(\bt^{\sum_{s=0}^{k-1}A^s}\big)_l\Big)^{b_l/\tau}\\
&= \tau\prod_{k=1}^\tau \Big(\prod_{l=1}^{d}\alpha_l^{b_l}\Big)^{\frac{\rho(A)^k}{\tau}}\Big(\prod_{l=1}^{d}\theta_l^{b_l}\Big)^{\frac{1}{\tau}\sum_{s=0}^{k-1}\rho(A)^s}\,\geq\, \tau\prod_{l=1}^{d}\alpha_l^{b_l},
\end{align*}
where we have used that $\bal\leq \ones$ because $\bu,\bv\in\Sn_+$. Thus, for all $ (j_1,\ldots,j_d)\in\J$ we have
$\prod_{l=1}^du_{l,j_l}^{b_l}> \prod_{l=1}^d(v_{l,j_l}\alpha_l)^{b_l}$,
a contradiction to the definition of $\bal$. 
Now, if $F(\bu)=\blam\krog\bu$ with $\blam\neq \ones$, then $\tilde F\in\NB^d$ defined as $\tilde F(\bx)=(\lambda_1^{-1},\dots,\lambda_d^{-1})\krog F(\bx)$ satisfies our assumptions and $\tilde F(\bu)=\bu$. So, if $(\bt,\bvb)\in\R^d_{+}\times(\Sn_{+}\sauf\kone_{++})$ satisfies $F(\bvb)=\bt\krog\bvb$, then $\tilde F(\bv)=(\theta_1/\lambda_1,\dots,\theta_d/\lambda_d)\krog\bv$ and thus 
$\prod_{l=1}^d \big(\theta_l/\lambda_l\big)^{b_l}<1$ implies $
\prod_{j=1}^d\theta_j^{b_j}<r(F).$
\end{proof}

Now let us fix $\bphi\in\kone_{++}$. Our next result is concerned with the uniqueness of positive eigenvectors in $\S^{\bphi}_{++}=\{\bx\in\kone_{++}\ | \ \ps{\bx_i}{\bphi_i}=1, \forall i\}$.  We first need to derive a number of intermediate results. The first one is a theorem with a flavor of fixed point theory in the sense that it only requires $G\colon \S^{\bphi}_{++}\to \S^{\bphi}_{++}$ to be nonexpansive under the metric $\mu_{\bb}$. The theorem states that if $G$ has two distinct positive eigenvectors $\bu,\bw\in\S_{++}^{\bphi}$, then $DG(\bu)$ has a fixed point $\bv$ which is orthogonal to $\bphi$.
The proof of this result can be easily obtained from the one of Theorem 6.4.1 \cite{NB}.
\begin{thm}\label{nonuniquethm}
Let $\bphi\in\kone_{++}$, $\bb\in\R^d_{++}$ and $G\colon\S^{\bphi}_{++}\to \S^{\bphi}_{++}$ be such that $\mu_{\bb}(G(\bx),G(\by)) \leq \mu_{\bb}(\bx,\by)$ for all $\bx,\by\in\S^{\bphi}_{++}$.
If there exist $\bu,\bw\in\S_{++}^{\bphi},\bu\neq\bw$ such that $G(\bu)=\bu$, $G(\bw)=\bw$ and $G$ is differentiable at $\bu$, then there exists $\bv\in V=\R^{n_1}\times\ldots\times\R^{n_d}, \bv\neq 0$ such that $\ps{\bv}{\bphi}=0$ and $DG(\bu)\bv=\bv$.
\end{thm}

 The second one is a lemma describing properties of $DF(\bu)$ and $DG(\bu)$ where $G\colon \S^{\bphi}_{++}\to \S^{\bphi}_{++}$ is defined in terms of $F$  as
\begin{equation}\label{defG}
G(\bx)=\big(\ps{\bphi_1}{F_1(\bx)}^{-1},\ldots,\ps{\bphi_d}{F_d(\bx)}^{-1}\big)\krog F(\bx).
\end{equation}
The lemma shows that when $\bu$ is a fixed point of $F\in\NB^d$ and $F$ is differentiable at $\bu$, then one can find $\t \bb\in\R^d_{+}$ such that $\t \bb\krog\bu$ is an eigenvector of $DF(\bu)$.
\begin{lem}\label{difflem}
Let $\bphi\in\kone_{++}$, $F\in\NB^d$, $A=\A(F)$ and $G$ as in \eqref{defG}. If there exists $\bu\in\Sphi_{++}$ with $F(\bu)=\bu$, $F$ is differentiable at $\bub$ and $\t\bb\in\R^d_{+,0}$ satisfies $A\t\bb=\t\bb$, then
$DF(\bu)\t\bu=\t\bu $ with $\t\bu = \t\bb\krog\bu.$
Moreover, $G$ is differentiable at $\bu$ and for every $\bz\in V$,
\begin{equation}\label{gradG}
DG(\bu)\bz = DF(\bu)\bz-\big(\ps{DF_1(\bu)\bz}{\bphi_1},\ldots,\ps{DF_d(\bu)\bz}{\bphi_d}\big)\krog\bu
\end{equation}
\end{lem}
\begin{proof}
For $F\in \NB^d$ let us write $D_kF_i(\bv)\in\R^{n_i\times n_k}$ to denote the Jacobian matrix of the mapping $\by_k\mapsto F_i(\bx_1,\ldots,\bx_{k-1},\by_k,\bx_{k+1},\ldots,\bx_d)$ at $\by_k= \bx_k$. 
By Lemma \ref{Eulerthm}, for all $k,i\in[d]$, we have $D_iF_k(\bu)\bu_i=A_{k,i}\bu_k$. Hence,
\begin{equation}\label{gradeval}
DF_i(\bu)(\bal \krog\bu)=(A\bal)_i \bu_i \qquad \forall \bal\in\R^d_{++}
\end{equation}
implying $DF(\bu)\t\bu=(A\t\bb)\krog\bu=\t \bu$. Now, if $F$ is differentiable at $\bxb\in\kone_{++}$, then
\begin{equation*}
D_kG_{i}(\bxb)=\frac{\ps{F_i(\bxb)}{\bphi_i}D_kF_{i}(\bxb)-F_{i}(\bxb)\bphib_i^\top D_kF_i(\bxb)}{\ps{F_i(\bxb)}{\bphi_i}^2}\qquad \forall k, i \in[d].
\end{equation*}
 In particular, if $\bxb=\bub\in\Sphi_{++}$ and $F(\bu)=\bu$, the above equation simplifies to
$D_kG_{i}(\bu)=D_kF_{i}(\bu)-\bu_{i}\bphi_i^\top D_kF_i(\bu).$
\end{proof}

We now state and prove Theorem \ref{unique} which extends Theorem 6.4.6 in \cite{NB} to the case $d\geq 1$. 
\newcommand{\tbphi}{\bar{\boldsymbol{\varphi}}}
\begin{thm}\label{unique}
Let $\bphi\in\kone_{++}$, $F\in\NB^d$ and $A=\A(F)$. Suppose that $A$ is irreducible, $\rho(A)=1$, there exist $\blam \in\R^d_{++}$ and $\bu\in\S_{++}^{\bphi}$ with $F(\bu)=\blam \krog \bu$ and $F$ is differentiable at $\bu$. Consider the linear mapping $L\colon\kone_+\to\kone_+$ defined as $L(\bx)=(\lambda_1^{-1},\dots,\lambda_d^{-1})\krog DF(\bu)\bx$ for every $\bx$, then $\rho(L)=1$ and if $\dim(\ker(I-L))=1$, then $\bu$ is the unique eigenvector of $F$ in $\S^{\bphi}_{++}$.
\end{thm}
\begin{proof}[Proof of Theorem \ref{unique}]
Let $\t\bb,\bb\in\Dn$ be such that $A\t\bb=\t\bb$ and $A^\top\bb = \bb$. These vectors always exist because $A$ is assumed to be irreducible. Suppose by contradiction that there exists $\bwb\in\Sphi_{++}\sauf\{\bu\}$ and $\t\blam\in\R^d_{++}$ such that $F(\bw)=\t\blam\krog\bw$. Let $\t F\in\NB^d$ be defined as in $\t F(\bx)=(\lambda_1^{-1},\dots,\lambda_d^{-1})\krog F(\bx)$ for every $\bx\in\kone_+$. Then, we have $\t F(\bu)=\bu$, $L=D\t F(\bu)$ and $\t F(\bw)=(\tilde \lambda_1/\lambda_1,\dots,\tilde \lambda_d/\lambda_d)\krog \bw$. Lemma \ref{difflem} implies that $\t\bu=\t\bb\krog\bu\in\kone_{++}$ satisfies $L\t\bu=\t\bu$. Theorem \ref{opcharac} implies that $L$ is a nonnegative matrix. Hence, Proposition \ref{Prop536} and $L\t\bu=\t\bu\in\kone_{++}$ imply that $\rho(L)=1$. Let $G$ be defined as \eqref{defG}, then $G$ is nonexpansive by Lemma \ref{contract}. By Theorem \ref{nonuniquethm}, there is a $\bv\neq 0$ with 
\begin{equation}\label{contreq}
\ps{\bv}{\bphi}=0,\qquad L\bv-\bal\krog\bu=\bv \quad \text{where}\quad\bal= \big(\ps{L\bv}{\bphi_1},\ldots,\ps{L\bv}{\bphi_d}\big).
\end{equation}
First, suppose that $\ps{\bb}{\bal}=0$. Then for $\tbphi\in\kone_{+,0}$ with $\ps{\bu_i}{\tbphi_i}=1$, $i\in[d]$, we have 
\begin{equation}\label{goodtrick}
\sum_{i=1}^d\ps{\big(L\bv\big)_i}{b_i\tbphi_i}= \sum_{i=1}^d\ps{\bv_i}{b_i\tbphi_i}+\sum_{i=1}^d\alpha_ib_i\ps{\bu_i}{\tbphi_i} = \sum_{i=1}^d\ps{\bv_i}{b_i\tbphi_i}.
\end{equation}
Let $(i,j_i)\in\I$ and define $\t\be^{(i,j_i)}\in\R^{n_i}_{+,0}$ as $\big(\t\be^{(i,j_i)}\big)_{l_i}=1$ if $j_i=l_i $ and $\big(\t\be^{(i,j_i)}\big)_{l_i}=0$ otherwise.
Furthermore, consider $\tbphi^{(i,j_i)}\in\kone_{+,0}$ defined as 
\begin{equation*}
\tbphi^{(i,j_i)}=\bigg(\frac{\ones}{\ps{\ones}{\bu_1}},\ldots,\frac{\ones}{\ps{\ones}{\bu_{i-1}}},\frac{\ones-\t\be^{(i,j_i)}}{\ps{\ones-\t\be^{(i,j_i)}}{\bu_i}},\frac{\ones}{\ps{\ones}{\bu_{i+1}}},\ldots,\frac{\ones}{\ps{\ones}{\bu_d}}\bigg).
\end{equation*}
Plugging $\tbphi^{(i,j_i)}$ into Equation \eqref{goodtrick} for every $(i,j_i)\in\I$ implies the existence of $M\in\R^{\t N\times \t N}$, with $\t N = n_1+\ldots+n_d$, such that $ML\bv=M\bv$,
$M_{(i,j_i),(k,l_k)}>0$ for every $(i,j_i),(k,l_k)\in\I$ with $(i,j_i)\neq (k,l_k)$ and  $M_{(i,j_i),(i,j_i)}=0$ for every $(i,j_i)\in\I$. In particular, $M$ is invertible and thus $L\bv=\bv$. Hence, by assumption, there exists $\beta\in\R\saufzero$ such that $\bv=\beta\t\bu$. We obtain the contradiction 
$0=\beta^{-1}\ps{\bv}{\bphi}=\sum_{i=1}^d\t b_i=1.$
Now, suppose that $\ps{\bb}{\bal}\neq 0$ and let $\norm{\cdot}$ be any monotonic norm on $\R^{n_1}\times \ldots\times \R^{n_d}$. Note that $A\bal \neq 0$ because it would imply the contradiction $0=\ps{A\bal}{\bb}=\ps{\bal}{A^\top\bb}=\ps{\bal}{\bb}.$
Let $\nu\in\N$, with \eqref{gradeval} and \eqref{contreq} we get
\begin{equation}\label{eq1contrad}
L^{\nu+1}\bv-\bv\ =\ \sum_{k=0}^{\nu} L^k(L\bv-\bv)\ =\ \sum_{k=0}^{\nu}L^{k}(\bal\krog\bu)\ =\ \sum_{k=0}^{\nu}(A^k\bal)\krog\bu.
\end{equation}
On the one hand, as $\t\bu>0$, there exists $t>0$ with $-t\t\bu\lek\bv\lek t\t\bu$. 
It follows that $ 0 \lek L^{\nu+1}\bv+t\t\bu\lek 2t\t\bu$ because $-t\t\bu \lek L^{\nu+1}\bv \lek t\t\bu$. Thus,
\begin{equation}\label{eq2contrad}
\norm{L^{\nu+1}\bv}\leq \norm{L^{\nu+1}\bv+t\t\bu}+\norm{t\t\bu}\leq 3t\norm{\bu} \qquad \forall \nu\in\N.
\end{equation}
On the other hand, as $A$ is irreducible, we know from Theorem 1.1 \cite{Francesco} that 
the sequence $\frac{1}{k+1}\sum_{s=0}^kA^s$ converges towards $\ps{\bb}{\t\bb}^{-1}\t\bb\bb^\top$ as $k\to \infty$.
This implies that we have 
$\lim_{\nu\to\infty}\norm{\sum_{k=0}^{\nu}(A^k\bal)\krog\bu}= \infty.$
A contradiction to \eqref{eq1contrad} and \eqref{eq2contrad}.
\end{proof}
In the above theorem, if $DF(\bu)$ is irreducible, then $L$ is irreducible and so $\dim(\ker(I-L))=1$ follows by the linear Perron-Frobenius theorem. Thus Theorem \ref{simplicity} follows as a consequence. However, on the contrary, note that there are cases where $\dim(\ker(I-L))=1$ is satisfied but $DF(\bu)$ is not irreducible (see e.g. \cite{NB} p. 143). 
\begin{proof}[Proof of Theorem \ref{simplicity}]
If $DF(\bu)$ is irreducible, then the assumptions on $DF(\bu)$ in Theorems \ref{rad<} and \ref{unique} are satisfied.
Hence, uniqueness of $\bu$ follows from Theorem \ref{Banachcombi} if $\rho(A)<1$ and Theorem \ref{unique} if $\rho(A)=1$. Finally, Theorem \ref{rad<} implies the second part of the claim.
\end{proof}

\section{Convergence to the unique positive eigenvector}\label{PM_section}
We conclude the paper with a study of the convergence of the iterates of a mapping $F\in\NB^d$ towards its unique positive eigenvector $\bu$. Such analysis is particularly interesting in applications as it naturally induces an algorithm for the computation of $\bu$ and $r_{\bb}(F)$. For example, this allows us to solve certain nonconvex optimization problems to global optimality \cite{gautier2016globally,tudisco2018core}, a hard task in general, or can be used to efficiently identify important components in networks with multiple layers \cite{frafra}. 

When $F$ is a strict contraction, convergence is a direct consequence of the Banach fixed point theorem, however when $F$ is nonexpansive we need stronger assumptions on $F$. For example, if $F\colon\R^2_+\to\R^2_+$ is the linear mapping $F(\bx)=M\bx$ with $M=\begin{psmallmatrix}0 & 1 \\ 1 & 0\end{psmallmatrix}$ then, 
although $M$ is irreducible, the iterates of $F$ will never converge towards its eigenvector. For the case $d=1$, it is proved in Theorem 2.3 \cite{Nussb} that the normalized iterates of a nonexpansive mapping $F\in\NB^1$ converge towards its positive eigenvector $\bu$ if $DF(\bu)$ is primitive. We prove in the following theorem that such a result can be extended for the case $d>1$. Furthermore, taking inspiration from the study of nonnegative multilinear forms (see e.g. \cite{Boyd,Chang_rect_eig,us,NQZ}), we show that each of the iterates induces two monotonic sequences which are particularly useful for the estimation of the spectral radius. These results are summarized in the following:
\begin{thm}\label{PM}
Let $F\in\NB^d$, $A=\A(F)$ and $\bb\in\Dn$. Suppose that $F$ has a positive eigenvector $\bu\in\Sn_{++}$ and define the sequence of normalized iterates given by $\bx^0\in\Sn_{++}$ and
\begin{equation*} 
\bx^{k}= \bigg(\frac{F_1(\bx^{k-1})}{\norm{F_1(\bx^{k-1})}_{\gamma_1}},\ldots,\frac{F_d(\bx^{k-1})}{\norm{F_d(\bx^{k-1})}_{\gamma_d}}\bigg) \qquad \forall k=1,2,\ldots
\end{equation*}
Then, $\lim_{k\to \infty} \bx^k=\bu$ if either $\rho(A)<1$ or $A^\top\bb = \bb$, $F$ is differentiable at $\bu$ and $DF(\bu)$ is primitive. Furthermore, if $A^\top \bb\leq \bb$, then 
\begin{equation*}
\widehat{\alpha}_k\ \leq \ \widehat{\alpha}_{k+1}\ \leq \ r_{\bb}(F)\ \leq \ \widecheck{\alpha}_{k+1}\ \leq \ \widecheck{\alpha}_{k} \qquad \forall k=0,1,2,\ldots
\end{equation*}
where 
$\widehat{\alpha}_k = \prod_{i=1}^d \mathfrak{m}_i\big(F(\bx^k)\big/\bx^k\big)^{b_i},$ $ \widecheck{\alpha}_k = \prod_{i=1}^d \mathfrak{M}_i\big(F(\bx^k)\big/\bx^k\big)^{b_i}$. Finally, if $A\bb <\bb$, then $\rho(A)<1$ and the following bound on the convergence rate holds
\begin{equation*}
\mu_{\bb}(\bx^k,\bu)\leq \rho(A)^{k}\,\frac{\mu_{\bb}(\bx^0,\bu)}{1-\rho(A)} \qquad \forall k\in\N.
\end{equation*}
\end{thm}
The proof of this theorem requires a number of preliminary results which we gather, together with the proof of the theorem itself, in the next final subsection.

\subsection{Convergence analysis}
First, we need the subsequent lemma which can be proved in the same way as Lemma 6.5.7 \cite{NB}, dealing with the case $d=1$.
\begin{lem}\label{lightlemma}
Let $F\in\NB^d$ and $\bu\in\kone_{++}$ with $F(\bu)=\bu$. If $F$ is differentiable at $\bu$ and $\nu$ is a positive integer such that $DF(\bu)^{\nu}$ has strictly positive entries, then 
$F^{\nu}(\bu)< F^{\nu}(\bx)$ for all $\bx\in\kone_{++}$ with $\bu\lekk\bx.$
\end{lem}
We recall known results of fixed point theory:
For $\bx\in\kone_{++}$ and $F\in\NB^d$, the orbit $\O(F,\bx)$ of $\bx$ under $F$ is defined as $\O(F,\bx)=\big\{ F^{k}(\bx)\mid k\in\N\big\} 
$. Furthermore, the $\omega$-limit set $\omega(F,\bx)$ of $\bx$ under $F$ is the set of accumulation points of $\O(F,\bx)$. For $F\in\NB^d$, Theorem 3.1.7 and Lemmas 3.1.2, 3.1.3 and 3.1.6 in \cite{NB} imply the following:
\begin{enumerate}[(I)]
\item If $F$ is nonexpansive with respect to the weighted Thompson metric $\bar\mu_{\bb}$ on $\kone_{++}$ and there exists $\bu\in\kone_{++}$ such that $\big(F^{k}(\bu)\big)_{k=1}^{\infty}\subset\kone_{++}$ has a bounded subsequence, then $\O(F,\bx)$ is bounded for each $\bx\in \kone_{++}$.\label{fixp1}
\item If $\bx\in\kone_{++}$ is such that $\O(F,\bx)$ has a compact closure, then $\omega(F,\bx)$ is a nonempty compact set and $F\big(\omega(F,\bx)\big)\subset\omega(F,\bx)$.\label{fixp2}
\item If $\bx\in\kone_{++}$ is such that $\O(F,\bx)$ has a compact closure and $|\omega(F,\bx)|=p$, then there exists $\bz\in\kone_{++}$ such that $\lim_{k\to\infty}F^{pk}(\bx)=\bz$ and $\omega(F,\bx)=\O(F,\bz)$.\label{fixp3}
\item If $F$ is nonexpansive with respect to $\bar\mu_{\bb}$, then for all $\bx\in\kone_{++}$ and $\by\in\omega(F,\bx)$, we have that $\omega(F,\by)=\omega(F,\bx)$.\label{fixp4}
\end{enumerate}
Property \eqref{fixp1} is a direct consequence of Calka's Theorem \cite{AleksanderCalka1984}. We are now ready to prove the following theorem which is a special case of Corollary 6.5.8 in \cite{NB} when $d=1$.
\begin{thm}\label{conv1}
Let $F\in\NB^d$, $\pmi{0}\in\Sn_{++}$ and $A=\A(F)$. Suppose that $\rho(A)=1$ and there exist $(\blam,\bu)\in\R^d_{++}\times\Sn_{++}$ such that $F(\bu)=\blam\krog\bu$. If $F$ is differentiable at $\bu$ and $DF(\bu)$ is primitive, then $\bu$ is the unique eigenvector of $F$ in $\Sn_{++}$ and the sequence $(\pmi{k})_{k=0}^{\infty}\subset\Sn_{++}$ defined in Theorem \ref{PM} satisfies $\lim_{k\to\infty} \pmi{k} = \bu$.
\end{thm}
\begin{proof}
First, note that the primitivity of $DF(\bu)$ implies that of $A$ by Lemma \ref{Eulerthm}. Hence, by Theorem \ref{unique}, $\bu$ is the unique positive eigenvector of $F$. Furthermore, there exist $\bb,\t\bb\in\Dn$ and $\nu\in\N$ such that $A^\top\bb=\bb$, $A\t\bb=\t\bb$ and $DF(\bu)^{\nu}>0$. Now, let $\blam\in\R^d_{++}$ with $F(\bu)=\blam\krog\bu$ and $\hat F\in\NB^d$ defined as $\hat F(\bx)=(\lambda_1^{-1},\dots,\lambda_d^{-1})\krog F(\bx)$. Then $\A(\hat F)=A$, $\bu$ is the unique eigenvector of $\hat F$, $\hat F$ is differentiable at $\bu$ and $D\hat F(\bu)^{\nu}>0$. 
 We show that for every $\bx\in\kone_{++}$, there exists $\bal\in\R^d_{++}$ such that $\omega(\hat F,\bx)=\{\bal\krog\bu\}$. Let $\bx\in\S^{\bphi}_{++}$ and consider the sequence
 $\xi_k= \prod_{i=1}^d\mathfrak{m}_i(F^k(\bx)\big/\bu)^{b_i}$. Then, we have 
 \begin{equation*}
 \prod_{i=1}^d\mathfrak{M}_i\big(\bx\big/{\bu}\big)^{b_i} \geq \xi_{k+1} \geq\prod_{i=1}^d\mathfrak{m}_i\big({F\big(\mathfrak{m}({F^{k}(\bx)}/{\bu})\krog\bu\big)}\big/{\bu}\big)^{b_i} =\xi_k
  \end{equation*}
 which implies that the sequence $(\xi_k)_{k=1}^\infty$ converges towards some $\xi>0$ as it is monotonic and bounded. In particular, it holds $\xi = \prod_{l=1}^d \bmini{l}{\bz}{\bu}^{b_l}$ for every $\bz\in\omega(\hat F,\bx)$.
Now, by Lemma \ref{contract}, we know that $\hat F$ is nonexpansive with respect to the weighted Thompson metric $\bar\mu_{\bb}$ on $\kone_{++}$. Since $\hat F(\bub)=\bub$, we have $\hat F^{k}(\bub)=\bub$ for every $k\in\N$ and thus \eqref{fixp1} implies that $\O(\hat F,\bxb)$ is bounded. Now, let $\nu\in\N$ be such that $DF(\bu)^{\nu}>0$. It follows from \eqref{fixp2}, that $\hat F^{\nu}\big(\omega(\hat  F,\bxb)\big)\subset\omega(\hat F,\bx)$ and thus $\hat F^{\nu}(\bz)\in\omega(\hat F,\bx)$ for every $\bz\in\omega(\hat F,\bx)$. Now, let $\bz\in\omega(\hat F,\bx)$ and suppose by contradiction that $\bz\neq\bbe\krog\bu$ for every $\bbe\in\R^d_{++}$. Then $\Bmin{\bz}{\bu}\krog\bu\lekk \bz$ and, with Lemma \ref{lightlemma}, we get 
$\Bmin{\bz}{\bu}^{A^\nu}\krog \hat F^{\nu}(\bu) = \hat F^{\nu}(\Bmin{\bz}{\bu}\krog\bu)\lekkk \hat F^{\nu}(\bz).$
Thus, with $\xi = \prod_{l=1}^d \bmini{l}{\bz}{\bu}^{b_l}$ and $\hat F^{\nu}(\bu)=\bu$, we obtain the contradiction
\begin{equation*}
\xi =  \prod_{l=1}^d\bmini{l}{\bz}{\bu}^{b_l}\bmini{l}{\hat F^{\nu}(\bu)}{\bu}^{b_l}< \prod_{l=1}^d\bmini{l}{\hat F^{\nu}(\bz)}{\bu}^{b_l}=\xi.
\end{equation*}
Hence, there exists $\bal\in\R^d_{++}$ such that $\bz=\bal\krog\bub$ and \eqref{fixp4} implies that $\omega(\hat F,\bx)=\omega(\hat F,\bal\krog\bu)$. As $A$ is primitive, we know from Theorem 1.1 \cite{Francesco} that it holds $
\lim_{k\to\infty}A^k = B$ with $B=\ps{\t\bb}{\bb}^{-1}\t\bb\bb^\top$.
In particular, we have
\begin{equation*}
\lim_{k\to\infty} \hat F^k(\bal\krog\bu)=\lim_{k\to\infty} \bal^{A^k}\krog \hat F^k(\bu)=\lim_{k\to\infty} \bal^{A^k}\krog \bu=\bxi^{B}\krog\bu.
\end{equation*}
Hence, we have $\omega(\hat F,\bx)=\omega(\hat F,\bal\krog\bu)=\{\bal^B\krog\bu\}$. So, $\lim_{k\to\infty}\hat F^{k}(\bxb)=\bal^B\krog\bub$ follows from \eqref{fixp3}. To conclude the proof, note that for every $ \by\in\kone_{++}$ and $i\in[d]$ it holds $\norm{\hat F_i(\by)}_{\gamma_i}^{-1}\hat F_i(\by)=\norm{F_i(\by)}_{\gamma_i}^{-1}F_i(\by)$ and
thus $\lim_{k\to\infty} \bx^k=\bu$.
\end{proof}
The following lemma generalizes Proposition 28 in \cite{us}. It implies the monotonicity of the sequence $(\widehat \alpha_k)_{k=1}^\infty$ and  $(\widecheck \alpha_k)_{k=1}^\infty$. 
\renewcommand{\NF}{\t G}
\begin{lem}\label{monothm}
Let $F\in\NB^d$ and $(\blam,\bub)\in\R^d_{++}\times\Sn_{++}$ be such that $F(\bu)=\blam\krog\bu$. Let $\bb\in\Dn$ with $A^\top\bb\leq \bb$, consider the mapping $\t G\colon\Sn_{++}\to\Sn_{++}$ defined as $\t G(\bx)=\big(\norm{F_1(\bz)}_{\gamma_1}^{-1},\ldots,\norm{F_d(\bz)}_{\gamma_d}^{-1}\big)\krog F$ 
and let $\cwl$, $\cwu$ be as in \eqref{def_cwu}. Then, for every $\bx\in\Sn_{++}$, it holds
$\cwl(F,\bx) \leq \cwl(F,\t G(\bx)) \leq r_{\bb}(F)\leq\cwu(F,\t G(\bx)) \leq  \cwu(F,\bx).$
\end{lem}
\begin{proof}
Let $\bx\in\Sn_{++}$, then $\mathfrak{m}(\NF(\bx)\big/\bx)\leq \ones$ because $\t G(\bx)\in\Sn_{++}$. Thus, with $\bs = \bb-A^\top\bb \in\R^d_{+}$, we have
$1 \leq  \prod_{i=1}^d\mathfrak{m}_i(\NF(\bx)\big/\bx)^{-\s_i}$.
It follows that
\begin{align*}
\cwl(F,\t G(\bx))  
&\geq \prod_{i=1}^d\norm{F(\bx)}_{\gamma_i}^{s_i}\bmini{i}{F(\Bmin{F(\bx)}{\bx}\krog\bx)}{F(\bx)}^{b_i} \\
& =\prod_{i=1}^d\norm{F(\bx)}_{\gamma_i}^{s_i}\bmini{i}{F(\bx)}{\bx}^{-s_i}\bmini{i}{F(\bx)}{\bx}^{b_i} \\
&=\prod_{i=1}^d\mathfrak{m}_i\big({\t G(\bx)}\big/{\bx}\big)^{-s_i}\bmini{i}{F(\bx)}{\bx}^{b_i} \geq \cwl(F,\bx).
\end{align*}
The inequality $\cwu(F,\t G(\bx))\leq \cwu(F,\bx)$ can be proved in the same way by swapping the inequalities and exchanging the roles of $\mathfrak{m}$ and $\mathfrak{M}$. The end of the proof follows from Theorem \ref{CW_intro}.
\end{proof}
\begin{proof}[Proof of Theorem \ref{PM}]
Let $\t G$ be defined as in Lemma \ref{monothm}. If $\rho(A)<1$, then by the proof of Theorem \ref{Banachcombi}, $\t G$ is a strict contraction with respect to $\mu_{\bb}$. In particular, $\lim_{k\to \infty}\bx^k=\bu$ and the linear convergence rate follows from the Banach fixed point theorem (see Theorem 3.1 \cite{pointfixe}). If $\rho(A)=1$, then $\lim_{k\to \infty}\bx^k=\bu$ follows from Theorem \ref{conv1}. Finally, if $\cwl$, $\cwu$ are defined as in Section \ref{CW_M_U}, then $\widehat \alpha_k=\cwl(F(\bx^k),\bx^k)$ and  $\widecheck \alpha_k=\cwu(F(\bx^k),\bx^k)$. Hence, the monotonicity of these sequences follows form Lemma \ref{monothm} and $\lim_{k\to \infty} \widehat \alpha_k=\lim_{k\to \infty} \widecheck \alpha_k=r_{\bb}(F)$ follows from the continuity of $\cwl,\cwu$.
\end{proof}

\section*{Acknowledgments}\addcontentsline{toc}{section}{Acknowledgments}
We are grateful to Shmuel Friedland and Lek-Heng Lim for a number of insightful discussions and for pointing out relevant references. We would also like to thank three anonymous referees for their careful reading of the manuscript and their very useful comments that largely improved the quality of the final manuscript.


%
\end{document}